\documentclass{amsart}
\usepackage{amssymb}
\usepackage{latexsym}
\usepackage{amsmath}
\usepackage{enumerate}
\usepackage{amsthm}
\usepackage{wasysym}
\usepackage{stmaryrd}
\usepackage{mathrsfs}
\usepackage[usenames]{xcolor}
\usepackage{tikz}
\usepackage{pifont}

\newcommand{\qee} {\hspace*{2mm}\hfill \ding{109}}

\newtheorem{theorem}{Theorem}[section]
\newtheorem{exa}[theorem]{Example}
\newenvironment{example}{\begin{exa} \rm}{\qee\end{exa}}
\newtheorem{exerc}[theorem]{Exercise}

\newtheorem{conj}[theorem]{Conjecture}

\newtheorem{ques}[theorem]{Open Question}
\newenvironment{question}{\begin{ques} \rm}{\qee\end{ques}}
\newtheorem{lem}[theorem]{Lemma}
\newenvironment{lemma}{\begin{lem} \it}{\end{lem}}
\newtheorem{cor}[theorem]{Corollary}
\newenvironment{corollary}{\begin{cor} \it}{\end{cor}}
\newtheorem{rem}[theorem]{Remark}
\newenvironment{remark}{\begin{rem} \rm}{\qee\end{rem}}

\newcommand{\medent}{\medskip\noindent}
 \newcommand{\tupel}[1]{{\langle #1 \rangle}}
\newcommand{\verz}[1]{\{ #1 \}}
\newcommand{\To}{\Rightarrow}
\renewcommand{\iff}{\leftrightarrow}

\newcommand{\hyph}{\mbox{-}}

\newcommand{\gnum}[1]{\underline{\ulcorner #1 \urcorner}}
\newcommand{\dpr}{\boxbar}
\newcommand{\gnumm}[1]{{\ulcorner #1 \urcorner}}

\DeclareMathOperator{\possible}{\text{\tikz[scale=.6ex/1cm,baseline=-.6ex,rotate=45,line width=.1ex]{
                            \draw (-1,-1) rectangle (1,1);}}}
\DeclareMathOperator{\necessary}{\text{\tikz[scale=.6ex/1cm,baseline=-.6ex,line width=.1ex]{
                            \draw (-1,-1) rectangle (1,1);}}}

\DeclareMathOperator{\odotpossible}{\text{\tikz[scale=.6ex/1cm,baseline=-.6ex,rotate=45,line width=.1ex]{
                            \draw (-1,-1) rectangle (1,1);  \draw (0,0) circle (.35);}}}

\newcommand{\apr}{{\vartriangle}}

\newcommand{\opr}{\necessary}

\newcommand{\oco}{\possible}

\title[The Absorption Law]{The Absorption Law\\ \emph{\small or} \\
How to Kreisel a Hilbert-Bernays-L\"ob}
\author{Albert Visser}
 \address{Philosophy, Faculty of Humanities,
                Utrecht University,
               Janskerkhof 13,
                3512BL~~Utrecht, The Netherlands}
\email{a.visser@uu.nl}
\date{\today}



\thanks{I am grateful to Fedor Pakhomov for sharing his ideas about 
slow provability and for a number of corrections
to an earlier version of this paper.
I thank Volodya Shavrukov for asking good questions and for pointing me to some 
references. I am grateful to Michael Beeson 
for pointing out a mistake and an ambiguity. I thank the anonymous referee for excellent suggestions..}

\begin{document}

\begin{abstract}
In this paper, we show how to construct for a given consistent theory $U$ a $\Sigma^0_1$-predicate that both satisfies the L\"ob Conditions and the Kreisel Condition ---even if $U$
 is unsound. We do this in such a way that $U$ itself can verify satisfaction of an internal version of the Kreisel Condition.
\end{abstract}

\maketitle

\section{Introduction}
When does a predicate $P$ in a theory $U$ count as a provability predicate for $U$? 
There are various ideas on the market
to explicate this notion. These ideas provide conditions for being a provability predicate that cater to various 
intuitions.\footnote{We refer the reader to \cite{halb:self14} 
and \cite{viss:seco16} for some philosophical discussion
of these ideas.} In the present paper, three classes of conditions will be considered: the Hilbert-Bernays-L\"ob 
Conditions, the Kreisel Condition and the Feferman Condition.
We will introduce the various conditions with some care in Section~\ref{condi}. In the present paper, we will 
not go into the philosophical discussion about the meaning 
of the conditions and their relative pro's and con's. However, in Appendix~\ref{examples}, we will give 
examples that illustrate that all three classes of conditions are
independent of one another. These examples can help the reader to form her own impression of what 
the conditions involve and, possibly, help her to get more grip 
on the issues surrounding the choice between the various classes of conditions.

The aim of our paper is to study the interplay of the  Hilbert-Bernays-L\"ob Conditions and the Kreisel 
Condition for the case of $\Sigma^0_1$-predicates.
The Kreisel condition for a provability predicate $\apr$ for a theory $U$ demands that
$U\vdash \apr A$ iff $U \vdash A$. A first question is whether we can have the Kreisel Condition for a predicate
that satisfies the L\"ob Conditions in case our theory $U$ is unsound. For example,
what happens when our theory is ${\sf PA}+{\sf incon}({\sf PA})$?  A second question is as follows.
If $\apr$ satisfies the L\"ob Conditions, the theory $U$, when consistent, cannot verify both the 
Kreisel Condition and the internal Kreisel Condition
$\apr\apr A \iff \apr A$. However, can we have the next best thing, to wit:
given an appropriately good provability predicate $\opr$ for $U$, can we find a predicate $\apr$, that satisfies
the L\"ob Conditions and for which we have both $U\vdash \apr A$ iff $U \vdash A$, and $U \vdash \opr\apr A \iff \opr A$?
As we will see, the answer to the last question is \emph{yes}. We can find, in many cases, a  predicate $\apr$ that
satisfies the L\"ob Conditions, the Kreisel Condition and the internal Kreisel Condition $U \vdash \opr\apr A \iff \opr A$.

We develop a general construction of a $\Sigma^0_1$-predicate $\apr$ that satisfies both the L\"ob Conditions and
the external and internal Kreisel Conditions from suitable data. The internal Kreisel
principle $\vdash \opr\apr A \iff \opr A$ splits in two sub-principles, to wit, \emph{the absorption law} $\vdash \opr\apr A \to \opr A$
and \emph{the emission law} $\vdash \opr A \to \opr\apr A$. Our main focus will be on the absorption law.

\subsection{Historical Note}
The following fact is due to Orey. Suppose $U$ is an extension of {\sf PA}. Then, we can find an elementary $\alpha$ such
that $\alpha$ represents the axiom set of $U$ over {\sf PA} and  $U \nvdash \opr_\alpha \bot$. See \cite{fefe:arit60}. 
See also  \cite[Chapter 2]{lind:aspe03}.

 A construction of a Fefermanian predicate\footnote{This means, roughly, that we use a standard arithmetization of syntax and proofs,
but tinker just with the representation of the axiom set. See Subsection~\ref{fefcon}, for a detailed explanation.}
 with $\Sigma^0_1$-axiomatization $\alpha$ for a theory $U$ that extends {\sf PA}, such that none of the
 iterated $\opr_{\alpha}$-inconsistency statements  $\opr^n_{\alpha}\bot$ is provable in $U$ is given in
 \cite{bekl:clas90}. See also \cite{kura:arit18a}.
  The result we will prove extends the results of Orey and Beklemishev.
 
 The predicates constructed in the present paper can be
 viewed as slow provability predicates. The absorption law holds for slow provability predicates over {\sf PA}.
  Slow provability over {\sf EA} was introduced and studied in \cite{viss:seco12}.
Slow provability over {\sf PA} was introduced and studied in \cite{frie:slow13}. Our knowledge concerning this notion was further extended in 
\cite{henk:slow16}, \cite{freu:slow17},  \cite{freu:shor20} and \cite{rath:long17}. 

The disadvantage of the present approach to slow provability is 
that the connection to proof theory and ordinal analysis is not visible. 
The advantage of the present approach to slow provability compared to the one of \cite{frie:slow13} is its wider scope.
Moreover, as we discuss below, it is not known whether the approach of \cite{frie:slow13} works for 
Heyting's Arithmetic {\sf HA}, the constructive counterpart of
{\sf PA}, where our approach works without problems. (This does not mean that it would not be very interesting to see how to extend the methods
of  \cite{frie:slow13} to the constructive case.)

An alternative approach to obtain a provability style predicate that satisfies both the L\"ob Conditions and the Kreisel Condition can be found in 
Section 5 of \cite{viss:tran16}. The approach in the present paper has a number of advantages. First, it is somewhat more perspicuous. Secondly,
the constructed predicates also satisfy the Hilbert-Bernays Condition. Thirdly, the construction of the predicates is fixed-point-free.
Fourthly, using the present approach we can also, in a number of cases, construct
predicates $\apr$ with the desired properties that are Fefermanian (in a sense that will be further specified in the paper).

The basic idea for the predicate constructed in this paper is due to
Fedor Pakhomov. He suggested considering this predicate when I asked him whether there was a non-model theoretic proof of
the absorption law for slow provability. However, the proof of absorption given in this paper is quite different from the one Fedor had in mind.

\subsection{Prerequisites}
The reader should be familiar with basic materials from \cite{haje:meta91}. For certain local results there
may be further prerequisites, but we will make these clear \emph{in situ}.

\subsection{Overview of the Paper}
In Section~\ref{baco}, we introduce the basic facts, notations and definitions for the rest of the paper.

Section~\ref{condi}, is a brief treatment of the L\"ob Conditions, the Kreisel Condition and the Feferman Condition.

Then, in Section~\ref{filo}, we construct a predicate $\apr$ with the desired properties for theories $U$ that extend Peano Arithmetic.
This is, in many respects, the simplest case of the construction. 
In this simple case, we have the extra property that our $\apr$ is Fefermanian. On the other hand,
the construction is subject to some restrictions. It
works only if we start with an elementary numeration $\alpha$ of the axiom set of $U$ such that $\pi\preceq \alpha$.
The construction delivers an axiom set numerated by $\widetilde\alpha$ for $U$. However, we will have $\pi\not\preceq \widetilde \alpha$,
so that the construction cannot be iterated. The construction in this section will also be covered by the more
general construction in Section~\ref{baarg}. However, we think the general treatment becomes easier to follow
if one has seen Section~\ref{filo} first.

In Section~\ref{baarg} , we give the basic construction in the abstract, starting from a predicate $\theta$ that satisfies a 
list of properties. We illustrate how $\opr_{\alpha_x}$ of Section~\ref{filo} does satisfy the properties for
$\theta$, so that the approach of Section~\ref{filo} is subsumed under the approach of Section~\ref{baarg}.

Finally, in Section~\ref{olijkesmurf}, we show, under fairly general conditions,  how to construct a predicate $\theta$ that
satifies the desired properties.

In Appendix~\ref{examples}, we provide separation examples between the L\"ob Condtions, the Kreisel Condition and the
Feferman Conditions.

\section{Basic Conventions, Notations, Definitions}\label{baco}
In this section, we introduce basic conventions and fix some notations and give some definitions.

\subsection{Theories}
A theory $U$ in this paper is a theory in the signature of arithmetic.\footnote{Everything in the paper lifts
to the more general case where a theory of arithmetic is interpretable in the given theory. However,
it is pleasant to avoid the extra notational burden of the more general case. The notational burdens of the
present paper seem to be sufficiently heavy already.} A theory is given by
a set $X$ of axioms. We will generally assume that $X$ is a recursively enumerable set.
However, $X$ is just given as a set and it is not intrinsically connected with a presentation.
We will assume as a default that $U$ extends Elementary Arithmetic {\sf EA}. 

Two salient theories of the paper are Elementary Arithmetic {\sf EA} and Peano Arithmetic {\sf PA}. The theory {\sf EA} is
$\mathrm I\Delta_0+{\sf exp}$. It is
finitely axiomatizable by a single axiom $B$. See \cite{haje:meta91}. The predicate $x= \gnum B$ will be called $\beta$.
The theory {\sf PA} has a standard elementary presentation $\pi$ of the axiom set corresponding to the usual axiom scheme.

We will also consider the extension of {\sf EA} with the $\Sigma^0_1$-collection principle $\mathrm B \Sigma^0_1$.
This principle is given by: \[\vdash \forall x \leq a \,\exists y\, S_0(x,y) \to \exists b\,  \forall x \leq a \,\exists y\leq b\, S_0(x,y).\]
Here $S_0$ is $\Sigma^0_1$ and may contain further parameters.

\subsection{Arithmetization}\label{arithm}
We will sometimes use implementation properties of the arithmetization like
monotonicity and the efficiency of syntactical operations. For this reason, we outline a
few features of the G\"odel coding we intend to use. We use a style of G\"odel numbering that
is due to Smullyan (see \cite{smul:theo61}). Our G\"odel numbering is based
on the length-first ordering. We enumerate the strings of our finite alphabet according to length
and the strings of the same length alphabetically. The G\"odel number of a string $s$ will be the number
of occurrence in this enumeration. In this ordering the arithmetical function tracing concatenation is of the order
of multiplication. We can use our bijective coding of strings to implement sequences of numbers. This has the bonus
that also concatenation of sequences of numbers will be of the order of multiplication.\footnote{Usually, there is some overhead
in defining sequences since we want to add some materials to make the definition of the projection function easy.
However, the uses of sequences to define syntax and proofs usually only require that we  can determine whether
something occurs in a sequence before something else. For this one does not need the extra material.}

We will in many cases employ modal notations. E.g., let ${\sf prov}_\alpha$ be the arithmetization of
provability from the axioms in $\alpha$. We write $\opr_\alpha A$ for ${\sf prov}_\alpha(\gnum{A})$. Here
$\gnum A$ is the numeral of the G\"odel number of $A$. We will sometimes
quantify the sentence-variables inside a modal operator. For example, we write things like:  
\[ (\dag) \;\;\; \forall A,B\, ((\opr_\alpha A \wedge \opr_\alpha(A \to B)) \to \opr_\alpha B).\]
This stands for: 
\[ (\ddag)\;\;\; \forall x\,\forall y\,\forall z\, (({\sf prov}_\alpha(x) \wedge {\sf imp}(x,y,z) \wedge  {\sf prov}_\alpha(z)) \to {\sf prov}_\alpha(y)).\]

\noindent
Admittedly, such notations are somewhat sloppy, but I think in practice they are very convenient. E.g., (\dag) is more pleasant to read
than (\ddag).

We employ the usual conventions for quantifiying numerical variables into modal contexts. E.g. $\opr_\alpha A(x)$
means $\exists z \,({\sf sub}(x,\gnum{v_0},\gnum{A(v_0)},z) \wedge {\sf prov}_\alpha(z))$.

We will employ the witness comparison notation. Suppose $A = \exists x\, A_0(x)$ and $B = \exists x\, B_0(x)$. We write:
\definecolor{uuxred}{cmyk}{0.2,1,0.9,0.1}
\begin{itemize}
\item
$A \leq B$ for $\exists x\, (A_0(x) \wedge \forall y < x\, \neg\, B_0(y))$. 
\item
$A < B$ for $\exists x\, (A_0(x) \wedge \forall y \leq x\, \neg\, B_0(y))$. 
\end{itemize}

\subsection{Ordering of Predicates for Axioms}
Let $\gamma(x)$ and $\delta(x)$ be formulas with only $x$ free that {\sf EA}-verifiably
represent classes of  arithmetical sentences. Let $T$ be an extension of {\sf EA}.
 We write $\gamma \preceq_T \delta$ for
\[T \vdash \forall A\, ({\sf prov}_\gamma(A) \to {\sf prov}_\delta (A)).\] Here ${\sf prov}_\alpha$ is a standard
arithmetization of provability from $\alpha$. Our default for $T$ will be {\sf EA} and we will write $\preceq$ for
$\preceq_{\sf EA}$.

It is easy to see that $\preceq_T$ is a partial pre-ordering. 

\section{Conditions for Provability Predicates}\label{condi}
In this section, we introduce three (classes of) conditions that aim to explicate when a predicate
is a provability predicate. 

\subsection{The L\"ob Conditions}
To state the L\"ob conditions we write $\apr A$ for $P(\gnum{A})$ and $\vdash$ for provability in $U$.
The L\"ob conditions (introduced in \cite{loeb:solu55}) are as follows. 
\begin{enumerate}[{\sf L}1.]
\item
If $\vdash A$, then $\vdash \apr A$
\item
$ \vdash (\apr A \wedge \apr(A\to B)) \to \apr B$
\item
$\vdash \apr A \to \apr\apr A$
\end{enumerate}
We obtain the  \emph{Hilbert-Bernays Conditions} in case we replace {\sf L}3 by:
\begin{enumerate}[{\sf HB}.]
\item
 $\vdash S \to \apr S$, for $\Sigma^0_1$-sentences $S$,
\end{enumerate}

\noindent
The usual assumption connected to the Hilbert-Bernays conditions is that $P$ be $\Sigma^0_1$, so that
{\sf L}3 is a special case of {\sf HB}. It is easy to see that if $P$ is not  $\Sigma^0_1$, we can have
${\sf L}1,2$ and {\sf HB} but not {\sf L}3. E.g., we may 
  take $P$ to be Feferman provability for {\sf PA}. 

We note that, in case $P$ is $\exists\Sigma^{\sf b}_1$, the L\"ob conditions are more general than the 
Hilbert-Bernays Conditions. For example, in a weak theory like
${\sf S}^1_2$ we do have the L\"ob Conditions for ${\sf S}^1_2$-provability 
 for a standard provability predicate ---assuming an efficient arithmetization---, but it is unknown whether we have the
Hilbert-Bernays Conditions.

Technically, the L\"ob Conditions constitute a superior analysis of the proof of the Second Incompleteness Theorem.
The philosophical use of the Conditions is independent of their technical interest. The 
philosophical idea is that the L\"ob Conditions explicate \emph{the theoretical role} that a provability predicate plays in a theory. 

We note that the L\"ob conditions do depend on the choice of G\"odel numbering and hence are still not entirely `coordinatefree'. 
For a study of this dependence and a proposal to
abstract away from it, see \cite{grab:inva18}.

The L\"ob Conditions also have a uniform and a global version. In the uniform version we allow parameters in the formulas inside the operator.
For example, {\sf L}2 becomes: $ \vdash \forall \vec x\, (\apr A(\vec x\,) \wedge \apr(A(\vec x\,)\to B(\vec x\,))) \to \apr B(\vec x\,)$. In the global version,
the quantifiers over sentences are not outside but inside the theory.  For example, {\sf L}2 becomes:
$ \vdash \forall A,B\in {\sf sent}\,( (\apr A \wedge \apr(A\to B)) \to \apr B)$. We note that the global version is stronger than the
uniform one. We will not consider the strengthened conditions in the present paper.

\subsection{The Kreisel Condition}
The \emph{Kreisel Condition} was first formulated in \cite{krei:prob53}. 
Its statement is as follows:
\begin{enumerate}[{\sf K}.]
\item
$U\vdash \apr A$ iff $U\vdash A$.
\end{enumerate}

\noindent We note that the Kreisel Condition is of a quite different nature than the L\"ob Conditions. It
just asks that the theory \emph{numerates} its own provability by the given predicate. 

One could imagine a variant of the Kreisel Condition where we just ask numerability in a base theory $U_0$ that is a sub-theory of $U$.

Finally, we observe that, like the L\"ob Conditions, the Kreisel Condition does depend on the chosen G\"odel numbering.

\subsection{The Feferman Condition}\label{fefcon}
We explain the idea that a provability-predicate is \emph{Fefermanian}. We derive this idea from the methodology introduced in \cite{fefe:arit60}.
The main ingredient of the idea is simply to fix a preferred arithmetization of provability and allow the
choice of the predicate $\alpha$ representing the axiom-set to be free, given that it satisfies certain adequacy conditions. 

The best way to present a Fefermanian predicate is to view it as a tuple $\tupel{U_0,U, \alpha}$. Here
$U_0$ is the \emph{base theory} and $U$ is the \emph{lead theory}. We ask that $U$ extends the base $U_0$.
We demand that $\alpha$ numerates an axiom set $X$ for $U$ in the base theory $U_0$.
In other words, we demand that $A\in X$ iff $U_0 \vdash \alpha(\gnum{A})$.

We note that the demands on a Fefermanian predicate treat the axioms of the lead theory via a condition similar
to the Kreisel Condition.

In the present paper, we will consider Fefermanian predicate modulo provability in the base theory.
Thus, we will say that $P$ is Fefermanian for $U$ over $U_0$ in the relaxed sense iff, there is an $\alpha$ such that
 $\tupel{U_0,U, \alpha}$ is Fefermanian in the strict sense and $U_0 \vdash \forall x\, (P(x) \iff {\sf prov}_\alpha(x))$.

The reader may object that the Feferman Condition does not count as a real condition since
it employs an unspecified specification of the arithmetization.\footnote{As remarked above the other conditions suffer, admittedly to
a lesser degree, from the same defect.} Of course, the reader is correct here. Feferman, in his paper,
does specify a choice for a proof system and an arithmetization. However, 
in  Feferman's arithmetization, the G\"odel number of a formula is superexponential in its length, so it is
not a convenient G\"odel numbering to work with within {\sf EA}.  
 Moreover, if Feferman's specific G\"odel numbering would really be the golden standard, it would be reasonable
 that everybody would know its specification, but, of course, that is not the case. I see the use of the Feferman idea more as
 dialogical. The reader is asked to take her favored good arithmetization in mind and read for {\sf prov} provability 
 according to that arithmetization. So, {\sf prov} becomes context dependent like the word `you'. I will employ the
 Feferman idea in this way.

\subsection{Properties of Fefermanian Predicates}

In this subsection we briefly consider some basic insights on Fefermanian predicates.

Let ${\sf A}_U$ be the class of all $\alpha$ in $\Sigma^0_1$ such that $\tupel{{\sf EA}, U,\alpha}$ is Fefermanian.

\begin{theorem}\label{grootgrutsmurf}
Let $U$ be a theory. 
Then ${\sf A}_U$ has a minimum w.r.t. $\preceq$
iff $U$ is finitely axiomatizable. 
\end{theorem}

\begin{proof}
Suppose $U$ is finitely axiomatizable, say by $A_0,\ldots, A_{n-1}$. Let  $\alpha^\ast(x) :=\bigvee_{i<n} x=\gnum{A_i}$.
Consider any $\alpha$ in ${\sf A}_U$. We find for $i<n$ that $U \vdash A_i$, and, hence ${\sf EA} \vdash \opr_\alpha A_i$.

We reason in {\sf EA}. 
Suppose $p$ witnesses $\opr_{\alpha^\ast} A$ and $p_i$, for $i<n$, witnesses $\opr_\alpha A_i$.
We obtain an $\alpha$-proof $q$ of $A$ by adding the $p_i$ `above' $A_i$ to $p$.
(Note that we do not need $\Sigma^0_1$-collection since $n$ is standard.)

Suppose $U$ is not finitely axiomatizable. Consider any $\alpha\in {\sf A}_U$. Clearly,
for any $n$ there is a $B$ such that $U\vdash B$ but the axioms in $\alpha$ that are $\leq n$ do
not prove $B$. Hence, $C:= \forall x\, \exists B \, (\opr_\alpha B \wedge \neg\, \opr_{\alpha_x} B)$ is true, 
where $\alpha_x(y) :\iff \alpha(y) \wedge y\leq x$.
Thus, ${\sf EA}+C$ is consistent. Let $\gamma (x) :\iff \beta(x) \vee x=\gnum{C}$, where $\beta$ is the standard axiomatization of {\sf EA}.
We define:
\[\alpha'(x) :\iff \alpha(x) \wedge \forall y \leq x\,\neg\,{\sf proof}_\gamma(x,\gnum{\bot}).\] 

\noindent
It is evident that $\alpha'\preceq \alpha$. Suppose (\dag) $\alpha\preceq \alpha'$. We reason inside ${\sf EA}+C$.
By (\dag), we have $\forall B\, (\opr_\alpha B \to \opr_{\alpha'} B)$.
Suppose $p$ is a $\gamma$-proof of $\bot$. It follows that the $\alpha'$ axioms are below $p$.
Consider $B$ such that $\opr_\alpha B$ but not $\opr_{\alpha_p}B$. It follows that $\neg\, \opr_{\alpha'}B$.
A contradiction. It follows that there is no $\gamma$-proof of $\bot$, in other words, $\oco_\gamma\top$.
We leave ${\sf EA}+C$.

We have shown ${\sf EA}+C\vdash \oco_\gamma\top$. But this contradicts the Second Incompleteness Theorem.
Hence (\dag) must fail.
\end{proof}

\begin{remark}
What happens if we replace {\sf EA}  in the definition of ${\sf A}_U$ by another base theory $T$ and, simultaneously,
 consider $\preceq_T$ in stead of $\preceq_{\sf EA}$? Inspection of the proof of Theorem~\ref{grootgrutsmurf}
 shows that we have to replace $\beta$
in the proof by an elementary predicate numerating the axioms of $T$ in $T$. This can always be arranged due to
Craig's trick. We also have to assume that $T$ is $\Sigma^0_2$-sound to be sure that $T+C$ is consistent.
We note that the application of the Second Incompleteness theorem goes through by the usual argument
since the analogue of $\gamma$ is elementary.\footnote{If we consider a $\Sigma^0_1$-axiomatization in a context without $\Sigma^0_1$-collection, 
the L\"ob conditions may fail. However, even in such L\"obless cases, the Second Incompleteness Theorem holds. See 
\cite{viss:look19}.}

Thus, our result goes through, as long as the base theory is $\Sigma^0_2$-sound. 
\end{remark}

\begin{theorem}
Consider theories $U_0$ and $U$ where {\sf EA} is a sub-theory of $U_0$ and $U_0$ is a sub-theory of $U$.
Suppose:
\begin{enumerate}[a.]
 \item
 $P$ numerates $U$ in $U_0$.
 \item
 $P$ contains $U_0$-provably all predicate-logical tautologies.
 \item
 $P$ is $U_0$-provably closed under finite conjunctions.
 \item
 $P$ is $U_0$-provably closed under modus ponens.
 \end{enumerate}
 Then, $P$ is Fefermanian \textup(in the relaxed sense\textup) for $U$ over 
 $U_0$ as witnessed by $\tupel{U,U_0, P}$.
\end{theorem}

\begin{proof}
Clearly, we have $U_0 \vdash \forall B\in P \, {\sf prov}_P(B)$. Conversely, reason in $U_0$.
Suppose $p$ is a $P$-proof of $B$. Let $X$ be the finite set of $P$-axioms used in $p$.
Then, $(\bigwedge X \to B)$ is a predicate logical tautology, so $(\bigwedge X \to B)\in P$. 
By closure under conjunction, we have
$\bigwedge X \in P$. Hence, by closure under modus ponens, we find $B\in P$.
\end{proof}

\begin{theorem}
Consider theories $U_0$ and $U$ where $\mathrm I\Sigma^0_1$ is a sub-theory of $U_0$ and $U_0$ is a sub-theory of $U$.
Let $P$ be a $\Sigma^0_1$-predicate.
Suppose:
\begin{enumerate}[a.]
 \item
 $P$ numerates $U$ in $U_0$.
 \item
 $P$ contains $U_0$-provably all predicate-logical tautologies.
 \item
 $P$ is $U_0$-provably closed under modus ponens.
 \end{enumerate}
 Then $P$ is Fefermanian for $U$ over $U_0$ with $P$ itself as representation of the axiom set.
\end{theorem}

\begin{proof}
Under the assumptions of the theorem, we can prove that $P$ is closed under finite conjunctions by $\Sigma^0_1$-induction.
\end{proof}

\begin{example}
We take as base and lead theory {\sf PA}. The predicate $\opr_\pi\opr_\pi$ is Fefermanian. Similarly, for
$\exists x\,\opr_\pi^{x+1}(\cdot)$. The last predicate is, modulo {\sf PA}-provable equivalence, 
\emph{Parikh provability} or \emph{fast provability}.
  Parikh provability can be obtained by adding to an axiomatization based on $\pi$ the Reflection Rule:
  ${\vdash \opr_\pi A} \To {\vdash A}$. See \cite{pari:exis71}. See also \cite{henk:nons16}.
\end{example}

\begin{theorem}
Suppose $U$ extends {\sf EA} and $P$ is Fefermanian w.r.t. a $\Delta_0({\sf exp})$-presentation $\alpha$ of the
axiom set. Then, $P$ satisfies the L\"ob Conditions.  
\end{theorem}

\begin{theorem}
Suppose $U$ extends ${\sf EA}+ \mathrm{B}\Sigma^0_1$ and $P$ is Fefermanian w.r.t. a $\Sigma_1$-presentation $\alpha$ of the
axiom set. Then, $P$ satisfies the L\"ob Conditions.  
\end{theorem}

\begin{theorem}\label{tuinsmurf}
Suppose $\tupel{U_0,U,\alpha}$ is a strict Fefermanian representation, where $\alpha$ is $\Sigma^0_1$, 
and suppose $U$ and $U_0$ are sound.
Then, $\opr_\alpha$ satisfies the Kreisel Condition for $U$.
\end{theorem}

\begin{proof}
Since $U_0$ is sound, we have $\alpha(\gnumm{A})$ iff $U_0 \vdash \alpha(\gnum{A})$.
So, $\alpha$ truly represents the axioms of $U$. It follows that $U \vdash A$ iff $\opr_\alpha A$.
Since $U$ is sound, we find
$\opr_\alpha A$ iff $U \vdash \opr_\alpha A$. So, we may conclude
$U\vdash A$ iff $U \vdash \opr_\alpha A$.
\end{proof}

Finally, we look into the interaction of elementary axiomatizations, i.e., $\Delta_0({\sf exp})$-axiomatizations,
and  $\Sigma^0_1$-axiomatizations. We note that an elementary formula numerates the same set in every
consistent theory. This immediately gives us the following insight.
\begin{theorem}
Suppose $U_0$ and $U_1$ are consistent subtheories of $U$. Suppose further that $\tupel{U_0,U,\alpha}$
is Fefermanian, where $\alpha$ is elementary. Then, $\tupel{U_1,U,\alpha}$ is Fefermanian.
\end{theorem} 
Similarly, a $\Sigma^0_1$-formula numerates the same set in all $\Sigma^0_1$-sound theories. So, we have:

\begin{theorem}
Suppose $U_0$ and $U_1$ are $\Sigma^0_1$-sound subtheories of $U$. Suppose further that $\tupel{U_0,U,\alpha}$
is Fefermanian, where $\alpha$ is $\Sigma^0_1$. Then, $\tupel{U_1,U,\alpha}$ is Fefermanian.
\end{theorem}

We remind the reader of Craig's trick. Let $U_0$ be a $\Sigma^0_1$-sound base theory.
Suppose $\sigma$ is a $\Sigma^0_1$-formula that $U_0$-provably represents a set
of arithmetical sentences. Suppose $\sigma(x) = \exists y\, \sigma_0(y,x)$, where $\sigma_0\in \Delta_0({\sf exp})$.
We define \[\hat\sigma(x) :\iff \exists y \leq x\, \exists z\leq x \, (\sigma_0(y,z) \wedge x= {\sf conj}({\sf id}({\sf num}(y),{\sf num}(y),z))).\]
Here {\sf conj} arithmetizes forming a conjunction, {\sf id} arithmetizes forming an identity statement
from terms, {\sf num} arithmetizes  the numeral function. Clearly, $\hat\sigma$ is elementary.
As is well known, we have: 
\begin{theorem}\label{labbekaksmurf}
\begin{enumerate}[i.]
\item
${\sf EA} \vdash \forall A\, (\opr_{\hat\sigma} A \to \opr_\sigma A)$.
\item
${\sf EA}+\mathrm B\Sigma_1 \vdash \forall A\, (\opr_{\hat\sigma} A \iff \opr_\sigma A)$.
\end{enumerate}
\end{theorem}

\begin{theorem}\label{kleutersmurf}
Suppose  $\tupel{U_0,U,\sigma}$ is Fefermanian and $U_0$ is $\Sigma^0_1$-sound and $\sigma$ is $\Sigma^0_1$.
Then, $\tupel{U_0,U,\hat\sigma}$ is also Fefermanian.
\end{theorem}

\begin{proof}
Let $X$ be the set of axioms numerated by $\sigma$ in $U_0$ and let $\hat X$ be $\hat\sigma$ in $U_0$.
By the $\Sigma^0_1$-soundness of $U_0$, the set $X$ is the set of numbers for which $\sigma$ is true
and the set $\hat X$ is the set  of numbers for which $\hat \sigma$ is true.
So, by the unformalized version of Theorem~\ref{labbekaksmurf}(ii), we find that both
$X$ and $\hat X$ axiomatize the same theory, to wit $U$.
\end{proof}

\noindent
We note that it is essential that $U_0$ is $\Sigma^0_1$-sound. If it were not the Craig construction
could transform a standard axiom to a non-standard axiom. The non-standard axiom would not be
visible in the numeration. 

\subsection{Examples}
We provide a list of examples for coincidence and separation of the conditions. 
As before $\beta$ is the standard representation of the axiom of {\sf EA} and $\pi$ is
the standard representation of the axioms set of Peano Arithmetic. We will, in our examples,
 prefer {\sf EA} over {\sf PA}, $\Sigma^0_1$-predicates over more complex ones,
and sound theories over unsound ones. Only in the first examples of Example~\ref{mpp} and Example~\ref{mmp},
perhaps, improvement is possible by finding an example that works for and over {\sf EA}.

\[
\begin{tabular}{|l||c|c|c||c|c|c|} \hline
& base & lead & $P$ & L\"ob & Kreisel & Feferman \\ \hline\hline
Example~\ref{ppp} & {\sf EA} & {\sf EA} & $\Sigma^0_1$ & $+$ & $+$ & $+$ \\ \hline
Example~\ref{ppm} &{\sf EA} & {\sf EA} & $\Sigma^0_1$ & $+$ & $+$ & $-$ \\ \hline
Example~\ref{pmp} & {\sf EA} & ${\sf EA}+\opr_\beta\bot$ & $\Sigma^0_1$ & $+$ & $-$ & $+$ \\ 
  & {\sf EA} & {\sf EA} & $\Sigma^0_2$ &  &  &  \\ \hline
Example~\ref{pmm} &{\sf EA} &{\sf EA} & $\Sigma^0_1$ & $+$ & $-$ & $-$ \\ \hline
Example~\ref{mpp} &{\sf PA} & {\sf PA} & $\Sigma^0_2$ & $-$ & $+$ & $+$ \\ 
& {\sf EA} & {\sf EA} & $\Sigma^0_{1,1}$ &&& \\ \hline
Example~\ref{mpm} &{\sf EA} & {\sf EA} & $\Sigma^0_1$ & $-$ & $+$ & $-$ \\ \hline
Example~\ref{mmp} & {\sf PA} & {\sf PA} & $\Sigma^0_2$ & $-$ & $-$ & $+$ \\ 
 & {\sf EA} & ${\sf EA}+\opr_\beta\opr_\pi\bot$ & $\Sigma^0_{1,1}$ &  &  &  \\ \hline
Example~\ref{mmm} & {\sf EA} & {\sf EA} & $\Sigma^0_1$ & $-$ & $-$ & $-$ \\ \hline
\end{tabular}
\]

\noindent
We will provide and verify the promised examples in Appendix~\ref{examples}. 

  
  
  
\section{Extensions of Peano Arithmetic}\label{filo}
Let $U$ be a consistent extension of {\sf PA} and let $\alpha$ be an elementary numeration of
an axiom set $X$ of $U$ in $U$, such that $\pi\preceq \alpha$.\footnote{Our argument also works 
under the weaker assumption that $\pi\preceq_{\sf PA} \alpha$.}
We will show how to construct a $\Sigma^0_1$-predicate
$\widetilde \alpha$ that numerates the  the axioms of $U$ in $U$.
Thus $\widetilde \alpha$ will be Fefermanian for $U$ over $U$.

The concrete examples to keep in mind are the standard representation $\pi$ of the axioms of {\sf PA} and
$\pi(x) \vee x = \gnum{\opr_\pi\bot}$ representing the axioms of ${\sf PA}+\opr_\pi\bot$.

We write $\alpha_x$ for the $\alpha$-axioms $\leq x$. So, $\alpha_x(y) :\iff \alpha(y) \wedge y\leq x$.
We write $\opr_{\alpha,(x)}$ for provability from $\alpha$ by a proof $\leq x$.\footnote{I use the round brackets to distinguish the intended notion from
$\opr_{\alpha,x}$ which is used in some of the literature for $\opr_{\alpha_x}$, where $\alpha_x(y) :\iff \alpha(x) \wedge y \leq x$.} 
We will use $S,S',\dots$ as variables ranging over $\Sigma^0_1$-sentences.

A number $x$ is \emph{small}, or $\mathcal S(x)$, iff $\opr_{\alpha,(x)}$ is \emph{$\Sigma^0_1$-reflecting}.
The means that $\mathcal S(x)$ iff $\forall S\,  (\opr_{\alpha,(x)} S \to {\sf true}(S))$, where
{\sf true} is a standard $\Sigma^0_1$-truth predicate. We note that smallness does depend on the chosen $\alpha$.
We also note that, by our assumptions on the G\"odel numbering, the quantifier over $S$ can be bounded by $n$.
It follows, by $\Sigma^0_1$-collection, that, modulo {\sf PA}-provability, smallness is a $\Sigma^0_1$-predicate.
Finally, smallness is clearly downward closed.

It is consistent with $U$ that not all numbers are small, since $U$ does not prove $\Sigma^0_1$-reflection for
$\alpha$-provability. On the other hand, for every $n$, we have that $U$ proves that it is small, i.e., $U \vdash \mathcal S(\underline n)$. 
The argument looks like this. (A more general argument is given in the proof of Lemma~\ref{technischesmurf}.) 
Consider a number $n$ and suppose that $n$ is small. 

Let $k\leq n$ and $s \leq n$.
 If $s$ is a code of a $\Sigma^0_1$-sentence $S$ and if  $k$ is the G\"odel number of an
$\alpha$-proof of $S$, then we have $U \vdash S$, and, hence,
 \[U \vdash ({\sf sent}_{\Sigma_1^0}(s) \wedge {\sf proof}_{\alpha}(k,s)) \to {\sf true}(s).\]
  If $s$ is not a code of a 
 $\Sigma^0_1$-sentence or, if $s$ is a code of a $\Sigma^0_1$-sentence $S$ and
   $k$ is not the G\"odel number of an
$\alpha$-proof of $S$, then we have $U \vdash \neg\, {\sf sent}_{\Sigma^0_1}(s) \vee \neg\, {\sf proof}_\alpha(\underline k,s)$, and, hence,
again,
 $U \vdash ({\sf sent}_{\Sigma_1^0}(s) \wedge {\sf proof}_{\alpha}(\underline k,s)) \to {\sf true}(s)$.
 It follows that:
 \[U \vdash \bigwedge_{k\leq n, s\leq n}(({\sf sent}_{\Sigma_1^0}(s) \wedge {\sf proof}_{\alpha}(\underline k,s)) \to {\sf true}(s)).\]
Hence, by $U$-reasoning, $U \vdash \mathcal S(\underline n)$.
The above reasoning is so simple that it can be verified in {\sf PA}, and so 
(\dag) ${\sf PA} \vdash \forall x\, \opr_\alpha  \mathcal S(x)$.
The principle (\dag) is a typical example of an \emph{outside-big-inside-small principle}. 
Objects that are very big in the outer world are  small in the inner world.

We define the slow provability of $A$ or $\apr A$ as: $A$ is provable from small $\alpha$-axioms. So,
\begin{itemize}
\item
$\widetilde \alpha(x) :\iff \alpha(x) \wedge \mathcal S(x)$
\item
$\apr A :\iff \opr_{\widetilde \alpha}A$
\end{itemize} 

\noindent
We list two formulas that are equivalent to $\apr A$ over {\sf PA} and all provide worthy ways 
of looking at it. Let $\opr^\ast_\alpha A :\iff \exists x\, \opr_{\alpha_x}A$.
\begin{itemize}
\item
$\apr A$ iff $\exists x\, (\opr_{\alpha_x}A \wedge \mathcal S(x))$,
\item
$\apr A$ iff $\opr^\ast A < \exists x\, \neg\, \mathcal S(x)$.
\end{itemize}

\begin{remark}
Our $\apr$ is part of a family of closely related predicates. To make this visible we consider a
slight variant of our $\apr$ with the same good properties.
We write $\opr_\alpha^{\Pi^0_1}$ for provability with a $\Pi^0_1$-oracle. 
We define $\apr^\circ A$ as $\opr_\alpha^\ast A < \opr_\alpha^{\Pi^0_1}\bot$. We note that this is equivalent to
$\exists x\, (\opr_{\alpha_x}A \wedge \oco_{\alpha,(x)}^{\Pi^0_1}\top)$

This representation brings out the analogy with Feferman provability which can be defined as 
$\opr_\alpha^\ast A < \opr_\alpha^\ast\bot$ and a provability predicate studied in \cite{henk:inte17} and \cite{henk:solo16}, to wit
$\exists x\, (\opr_{\alpha_{x+1}}A \wedge \oco^{\Pi^0_1}_{\alpha_x}\top)$ or, alternatively,
$\opr_\alpha^\ast A \leq \opr_\alpha^{\ast\Pi^0_1}\bot$. 
We note that, unlike $\apr^\circ$, these predicates are not $\Sigma^0_1$.
\end{remark}

\noindent 
Suppose $A$ is in $X$. Then, $U \vdash \alpha(\gnum A)$. Since also $U \vdash \mathcal S(\gnum A)$, we find
$U \vdash \widetilde \alpha(\gnum A)$. Conversely, suppose $U \vdash \widetilde \alpha(\gnum A)$. Then,
$U \vdash \alpha(\gnum A)$ and, hence, $A \in X$.
Thus $\widetilde \alpha$ numerates $X$ in $U$.

Since $\apr$ is $\opr_{\widetilde\alpha}$ and $\widetilde\alpha$ is $\Sigma^0_1$, it follows, in {\sf PA} by $\Sigma^0_1$-collection, that
that $\apr$ is $\Sigma^0_1$. Hence, $\apr$ satisfies the L\"ob conditions.

We show that {\sf PA} verifies emission and absorption for $\apr$. By the soundness of {\sf PA}, the Kreisel Condition follows.

We first prove emission. We prove the stronger ${\sf PA}\vdash \opr_\alpha A \to \opr_\alpha \apr A$. 
We reason in {\sf PA}. Suppose $\opr_\alpha A$. Then, clearly, for some $x$, we have $\opr_{\alpha_x} A$.
Hence, $\opr_\alpha \opr_{\alpha_x}A$. Also, (\dag) gives us $\opr_\alpha\mathcal S(x)$.
So, $\opr_\alpha (\opr_{\alpha_x}A \wedge \mathcal S(x))$ and, thus, $\opr_\alpha\apr A$.
 
We prove absorption.  The proof turns out to be remarkably simple. 
We find $R$ such that ${\sf EA} \vdash R \iff  (\exists x \, \opr_{\alpha_x} A ) < \opr_\alpha R$.
 We note that  $R$ is $\Sigma^0_1$.

We reason in {\sf PA}. 
Suppose $\opr_\alpha\apr A$. We prove $\opr_\alpha A$.
We reason inside $\opr_\alpha$. Since, by assumption,  $\apr A$, 
 we have, for some $x$, (i) $\opr_{\alpha_x} A$ and (ii) $\forall S\, (\opr_{\alpha,(x)} S \to {\sf true}(S))$.
In case not $\opr_{\alpha,(x)} R$, by (i), we find $R$. If we do have $\opr_{\alpha,(x)} R$, we find $R$ by (ii).
 We leave the $\opr_\alpha$-environment.
We have shown $\opr_\alpha R$. 
It follows, (a) that for some $p$, we have $\opr_\alpha \opr_{\alpha,(p)} R$ and, by the fixed point equation 
for $R$, (b) $\opr_\alpha ((\exists x\, \opr_{\alpha_x} A) < \opr_\alpha R)$. 
Combining (a) and (b), we find $\opr_\alpha\opr_{\alpha_p} A$, and, thus,
since $U$, as axiomatized by $\alpha$, is, {\sf EA}-verifiably, essentially reflexive by our assumption that
$\pi\preceq \alpha$, we obtain  $\opr_{\alpha} A$, as desired. 
We leave {\sf PA}. We have shown ${\sf PA}  \vdash \opr_\alpha\apr A \to \opr_\alpha A$.

\medskip
What happens if we drop the assumption that $\alpha$ is elementary and work with a $\Sigma^0_1$-predicate $\sigma$?
We have the following. 

\begin{theorem}\label{sluwesmurf}
Suppose $\tupel{{\sf PA},U,\sigma}$ is Fefermanian, where $\sigma$ is $\Sigma^0_1$ and 
$\pi\preceq \sigma$. Then we can construct a $\Sigma^0_1$-predicate
$\sigma^\ast$ such that $\tupel{U,U,\sigma^\ast}$ is Fefermanian and $\opr_{\sigma^\ast}$ satisfies, in $U$, the L\"ob-conditions, the
Kreisel condition and the pair $\opr_{\sigma^\ast}$, $\opr_{\sigma}$ satisfies emission and absorption over $U$.
\end{theorem}

\begin{proof}
We take $\sigma^\ast$ to be $\widetilde{\hat\sigma}$. We note that $\tupel{{\sf PA},U,\hat\sigma}$ is Fefermanian and thus
$\opr_{\sigma^\ast}$ satisfies, in $U$, the L\"ob-conditions and the Kreisel condition. Moreover the pair  
$\opr_{\sigma^\ast}$, $\opr_{\hat\sigma}$ satisfies emission and absorption over $U$. However,
$\opr_\sigma$ and $\opr_{\hat\sigma}$ are co-extensional over {\sf PA} and, hence, \emph{a fortiori},
over $U$. So, the pair  
$\opr_{\sigma^\ast}$, $\opr_{\sigma}$ also satisfies emission and absorption over $U$
\end{proof}
  
 \begin{remark}
The arguments of this section can be extended to constructive logic. In this case we still have the representations
$\beta$ for the axiom set of $\mathrm i\hyph {\sf EA}$ and $\pi$ for the axiom set of {\sf HA}. So the whole development remains unchanged.
One just has to check that never a step was taken that is essentially classical. 

The intuitionistic development has an important point. In their paper \cite{arde:sigm18}, Mohammad Ardeshir and Mojtaba Mojtahedi characterize the
provability logic of {\sf HA} for $\Sigma^0_1$-substitutions. This is the most informative result on the provability logic of {\sf HA} at the moment of writing.
An alternative proof has been developed in \cite{viss:prov19}.
This proof uses slow provability in the style of Friedman, Rathjen {\&} Weiermann for {\sf HA}. The proof works because only a restricted version 
of the absorption law is needed. The validity of the full absorption law is plausible but not proved. Replacement by of 
Friedman-Rathjen-Weiermann slow provability by slow provability in the
style of the present paper (as suggested by Fedor Pakhomov) does give us full absorption.

We show that we get a strengthened version of absorption in the case of {\sf HA}. The proof is intended for readers with
some background in the metamathematics of constructive arithmetical theories.

\begin{theorem}
$\mathrm i \hyph {\sf EA} \vdash \opr_\pi (A\vee B) \iff \opr_\pi(A \vee \opr_{\widetilde \pi} B)$.
\end{theorem}

\begin{proof}
We reason in i-{\sf EA}.

Suppose $\opr_\pi (A\vee B)$. It follows by either q-realizability or the de Jongh translation that, for some $x$, we have
$\opr_\pi(A \vee \opr_{\pi_x} B)$. From this, we may conclude $\opr_\pi (A \vee \opr_{\widetilde \pi} B)$.

Conversely, suppose $\opr_\pi (A \vee \opr_{\widetilde \pi} B)$. By the left-to-right case (with change of variables),  we have 
$\opr_\pi (\opr_{\widetilde \pi} A \vee \opr_{\widetilde \pi} B)$.
Hence, $\opr_\pi \opr_{\widetilde \pi} (A \vee B)$. So, by absorption, $\opr_\pi(A \vee B)$.
\end{proof}
\noindent 
Thus, the alternative predicates that satisfy the absorption law also have a rich constructive life. 
\end{remark}
 
 \section{The Abstract Construction}\label{baarg}
 In this section we present a construction that builds an appropriate $\apr$ from a given predicate 
 $\theta$ that satisfies certain good properties. 
 
 As before the variables $S,S',\dots$ will range over (codes of) $\Sigma^0_1$-sentences.
 
 \subsection{The Argument}
 Let $U$ be a theory. Suppose $\alpha(x)$ is an elementary predicate that numerates the axioms of $U$ in $U$.
 Let $\theta(y,z)$ be a $\Sigma^0_1$ binary predicate. We demand that $\theta$ is {\sf EA}-verifiably, upwards persistent in $y$, i.e., we 
 assume  that \[{\sf EA} \vdash (\theta(y,z) \wedge y < y') \to \theta(y',z).\]
 Let  $\dpr_{\theta,y}A$ be $\theta(y,\gnum{A})$. 
 We write $\dpr_yA$ as long as $\theta$ is given in the context.
 
 As a heuristic, the reader may think of $\dpr_{\theta,y}A$ as a generalization of $\opr_{\alpha_y}A$ as we used it
 in Section~\ref{filo}, studying the
 case where $\pi\preceq \alpha$.
 
 We define:
\begin{itemize}
\item
{\sf true} is the $\Sigma^0_1$-truth predicate, which is of the form  $\exists y\,{\sf true}_0(y,x)$, where ${\sf true}_0$ is
$\Delta_0({\sf exp})$. We write ${\sf true}^z(x)$ for $\exists y \leq z\,{\sf true}_0(y,x)$.
\item
$\opr_{\alpha,(x)} A :\iff \exists p\leq x \,{\sf proof}_\alpha(p,\gnum{A})$, where {\sf proof} is the standard
arithmetization of the proof predicate.
\item
$\mathcal S(x) :\iff \exists z\, \forall S \leq x\, (\opr_{\alpha,(x)} S \to {\sf true}^z(S))$.
Here the variable `$S$' ranges over $\Sigma^0_1$-sentences.
\item
$\apr_\theta A : \iff  \exists x (\dpr_{\theta,x} A \wedge \mathcal S(x))$. 
We will usually write $\apr$ for $\apr_\theta$ suppressing the contextually given $\theta$.
We note that modulo some rewriting $\apr_\theta$ is $\Sigma^0_1$.
\end{itemize}
The definition of $\apr_\theta$ is in essence due to Fedor Pakhomov.

As explained in Subsection~\ref{arithm}, we assume that we have a reasonable coding of proofs in which the code
of the proof is larger than the code of the conclusion. We fix, for the moment $\theta$ in the background. 

We note that our definition of $\mathcal S$ is slightly different from the one in Section~\ref{filo}. This is just to compensate
for the lack of $\Sigma^0_1$-collection.
We have:
\[
\begin{tabular}{lrrcll}
& ${\sf EA}$ & $\vdash$ & $\mathcal S(x)$ & $\to$ & $\forall S\, (\opr_{\alpha,(x)} S \to {\sf true}(S))$\\
& ${\sf EA} + \mathrm{B}\Sigma_1$ &  $\vdash$ & $\mathcal S(x)$ & $\iff$ & $\forall S\, (\opr_{\alpha,(x)} S \to {\sf true}(S))$\\
(\dag) & ${\sf EA}$ &   $\vdash$ & $\apr A$  & $\to$ &  $\exists x\, (\dpr_x A \wedge \forall S\, (\opr_{\alpha,(x)} S \to {\sf true}(S)))$\\
& ${\sf EA} + \mathrm{B}\Sigma_1$ &  $\vdash$ & $ \apr A$ & $\iff$ &  $\exists x\, (\dpr_x A \wedge \forall S\, (\opr_{\alpha,(x)} S \to {\sf true}(S)))$
\end{tabular}
\]

\noindent 
We note that we can write the right-hand-side of (\dag) as: \[ (\exists x\, \dpr_x A) < (\exists x\, \exists S\, (\opr_{\alpha,(x)} S \wedge \neg\, {\sf true}(S))).\]
Here the witness comparison is only concerned with the outer quantifiers.

\begin{lemma}\label{grotesmurf}
 ${\sf EA} + \forall x\, (\opr_\alpha\dpr_x A \to \opr_\alpha A) \vdash \opr_\alpha\apr A \to \opr_\alpha A$.
\end{lemma} 

\begin{proof}
We find $R$ such that ${\sf EA} \vdash R \iff  (\exists x \, \dpr_x A ) < \opr_\alpha R$.
 We note that  $R$ is $\Sigma^0_1$.

We reason in ${\sf EA} + \forall x\, (\opr_\alpha \dpr_x A \to \opr_\alpha A)$. Suppose $\opr_\alpha\apr A$. We prove $\opr_\alpha A$.

We reason inside $\opr_\alpha$. Since, by assumption,  $\apr A$, 
 we have, for some $x$, (i) $\dpr_x A$ and (ii) $\forall S\leq x\, (\opr_{\alpha,(x)} S \to {\sf true}(S))$.
In case not $\opr_{\alpha,(x)} R$, by (i), we find $R$. If we do have $\opr_{\alpha,(x)} R$, we find $R$ by (ii).\footnote{We
note that, in this step, we use  that
(the code of) the conclusion is smaller than (the code of) the proof.}
 We leave the $\opr_\alpha$-environment.

We have shown $\opr_\alpha R$. 
It follows, (a) that for some $p$, we have $\opr_\alpha \opr_{\alpha,(p)} R$ and, by the fixed point 
equation for $R$, (b) $\opr_\alpha ((\exists x\, \dpr_x A) < \opr_\alpha R)$. 
Combining (a) and (b) and the upward persistence of $\dpr_x$, we find $\opr_\alpha\dpr_{p} A$, and, thus,  $\opr_\alpha A$, as desired. 
We leave ${\sf EA} + \forall x\, (\opr_\alpha \dpr_x A \to \opr_\alpha A)$.

We have shown ${\sf EA} + \forall x\, (\opr_\alpha \dpr_x A \to \opr_\alpha A) \vdash \opr_\alpha\apr A \to \opr_\alpha A$.
\end{proof}

\noindent The proof of Lemma~\ref{grotesmurf} deserves a few comments.

\begin{remark}
We note that the argument also works when we define $\apr A$ as $ \exists x\, (\dpr_x A \wedge \forall S\, (\opr_{\alpha,(x)} S \to {\sf true}(S)))$.
The argument does not use that $\apr$ is $\Sigma^0_1$.
\end{remark}

\begin{remark}
In all applications of Lemma~\ref{grotesmurf}, {\sf EA} verifies not just the principle $\forall x\, (\opr_\alpha\dpr_x A \to \opr_\alpha A)$ for the concrete
choice of $\dpr$ of the application, but the stronger  $\forall x\, \opr_\alpha(\dpr_x A \to  A)$. We note that using this last condition, we may obtain
the theorem without the demand that $\dpr_y$ is upward persistent in $y$. In
${\sf EA}+\forall x\, \opr_\alpha(\dpr_x A \to  A)$, we can go from $\opr_\alpha \opr_{\alpha,(p)} R$ and $\opr_\alpha ((\exists x\, \dpr_x A) < \opr_\alpha R)$
to $\opr_\alpha\bigvee_{z<p} \dpr_z A$, and, hence, $\opr_\alpha A$.
\end{remark}

\begin{remark}
The proof of Lemma~\ref{grotesmurf} does not use exponentiation and would work in ${\sf S}^1_2$. The reason is that we only use
${\sf true}(R) \to R$, which is the direction of ${\sf true}(R) \iff R$ that does not require exponentiation.
\end{remark}

\begin{remark}
Let $\mathrm i\hyph{\sf EA}$ be the constructive version of {\sf EA}. Let  $U$ be a constructive theory that extends
$\mathrm i\hyph{\sf EA}$. Suppose $\mathrm i\hyph{\sf EA} \vdash (\theta(y,z) \wedge y < y') \to \theta(y',z)$.

Then, inspection shows that the entire proof of Lemma~\ref{grotesmurf} also works when we substitute $\mathrm i\hyph{\sf EA}$
for {\sf EA}. This uses the basic insight that $\opr_{\alpha,(x)}R$ is decidable in $\mathrm i\hyph{\sf EA}$. So the case-splitting in the
proof can be constructively justified.

Thus, we find  $\mathrm i\hyph{\sf EA} + \forall x\, (\opr_\alpha\dpr_x A \to \opr_\alpha A) \vdash \opr_\alpha\apr A \to \opr_\alpha A$.
\end{remark}

\noindent We prove the outside-big-inside-small lemma for $\mathcal S$ as a notion of smallness in {\sf EA}.
The proof has to be a bit more elaborate that in the luxurious case where we had full {\sf PA} to work with.

\begin{lemma}\label{technischesmurf}
${\sf EA} \vdash \forall x\, \opr_\alpha \mathcal S(x)$.
\end{lemma}

\begin{proof}
We work in {\sf EA}. We prove the desired result by induction on $x$. We need a multi-exponential bound for
the $\opr_\alpha$-proofs. We will extract the desired bound by inspecting the induction step.

The base case is trivial since there will be no $S \leq 0$. The proof witnessing the base will be given by a
standard number $\underline n$. 

We turn to the induction step. Suppose $p_0$ witnesses   $\opr_\alpha\mathcal S(x)$.
We have two possibilities: ${\sf proof}_\alpha(x+1,S^\ast)$, for some $S^\ast \leq x+1$, or $\neg\, {\sf proof}_\alpha(x+1,S^\ast)$, for all $S^\ast\leq x+1$.

Suppose ${\sf proof}_\alpha(x+1,S^\ast)$.
 Inspecting the proof of the truth-lemma for ${\sf true}$ in \cite[Ch V, Section 5b, pp361--366]{haje:meta91}, we obtain 
 a proof code $p_1$ such that ${\sf proof}_\alpha(p_1,S^\ast \to {\sf true}(S^\ast))$. The transformation $S^\ast \mapsto p_1$ is p-time.
 By \cite[Ch III, Lemma 3.14, p175]{haje:meta91},  we obtain an $\alpha$-proof $p_2$ of  ${\sf proof}_\alpha(x+1,S^\ast)$.
 The transformation $x+1 \mapsto p_2$ is of order $2_{\underline k}^{x+1}$, where $k$ is standard and the subscript $\underline k$ 
 indicates the number of iterations of exponentiation. 
 Working inside $\opr_\alpha$ we can put these facts together to obtain 
 \[\text{(a) } \mathcal S(x)\text{, (b) }{\sf proof}_\alpha(x+1,S^\ast)\text{ and (c) }{\sf true}(S^\ast).\]
 Let $z_0$ be the witness of (a), let $z_1$ be the witness of (c). Then, it is easily seen that $z := {\sf max}(z_0,z_1)$ witnesses
  $\mathcal S(x+1)$.
  
  Suppose $\forall S^\ast\leq x+1 \, \neg\, {\sf proof}_\alpha(x+1,S^\ast)$. By \cite[Ch III, Lemma 3.14, p175]{haje:meta91}, we may find an $\alpha$-proof $p_3$ 
  of $\forall S^\ast\leq x+1 \, \neg\, {\sf proof}_\alpha(x+1,S^\ast)$ 
  where the transformation $x+1 \mapsto p_3$ is of order $2_{\underline k}^{x+1}$. Using \[\text{(d) }\forall S^\ast\leq x+1 \, \neg\, {\sf proof}_\alpha(x+1,S^\ast)\]
  inside $\opr_\alpha$, we easily find the desired proof of $\mathcal S(x+1)$.
  
  We note that apart from a bit of overhead we extend $p_0$ with at most two proofs that are estimated by $2_{\underline k}^{x+1}$. So,
  the resulting proof will be of order $p_0\times (2_{\underline k}^{x+1})^2$. Thus, after all is said and done, the proof we obtain of
  $\mathcal S(x+1)$ will be estimated by $\underline n \times (2_{\underline k}^{x+1})^{2(x+1)} =
  2^{2_{\underline k -1}^{x+1}\times 2(x+1)}
  \leq  2_{\underline {k}}^{2x+1}$, assuming that $k \geq 2$.
\end{proof}

\begin{lemma}\label{stoeresmurf}
 ${\sf EA}  \vdash  \forall x\, \opr_\alpha(\dpr_x A \to \apr A)$. Hence, 
 \[{\sf EA} + (\opr_\alpha A \to \exists x\, \opr_\alpha \dpr_x A) \vdash \opr_\alpha A \to \opr_\alpha \apr A.\]
\end{lemma}

\begin{proof}
We work in {\sf EA}. Let $x$ be given. 
By Lemma~\ref{technischesmurf}, we
find $\opr_\alpha\mathcal S(x)$. Thus,  $\opr_\alpha(\dpr_x A \to (\dpr_x A \wedge\mathcal S(x)))$. This
gives us $\opr_\alpha(\dpr_x A \to \apr A)$.
\end{proof}

\begin{lemma}\label{babysmurf}
 ${\sf EA} + \forall x \, ((\dpr_x A \wedge \dpr_x(A\to B)) \to \dpr_x B)   \vdash  (\apr A \wedge \apr(A\to B)) \to \apr B$.
\end{lemma}

\begin{proof}
We work in  ${\sf EA} + \forall x \, ((\dpr_x A \wedge \dpr_x(A\to B)) \to \dpr_x B) $. Suppose $\apr A$ and $\apr (A\to B)$.
It follows that, for some $x$, we have $\dpr_x A$ and $\mathcal S(x)$ and that, for some $y$, we have $\dpr_y (A\to B)$ and
$\mathcal S(y)$. Let $z := {\sf max}(x,y)$. It is easily seen that $\dpr_z A$ and $\dpr_z (A\to B)$ and $\mathcal S(z)$. Hence,
$\dpr_z B$ and $\mathcal S(z)$, and, thus, $\apr B$.
\end{proof}

\begin{lemma}\label{minismurf}
 ${\sf EA} +  \forall S\, \exists x \, \opr_\alpha(S \to \dpr_x S) \vdash \opr_\alpha (S \to \apr S)$.
\end{lemma}

\begin{proof}
This is immediate by Lemma~\ref{technischesmurf}.
\end{proof}

\noindent
We formulate the obvious theorem that follows from the Lemmas.
Let ${\sf W}_{\alpha,\theta}$ be {\sf EA} plus the following principles:
\begin{enumerate}[a.]
\item
$ \forall x\, (\opr_\alpha\dpr_x A \to \opr_\alpha A)$
\item
$\opr_\alpha A \to \exists x\, \opr_\alpha \dpr_x A$ 
\item  
$ \forall x \, ((\dpr_x A \wedge \dpr_x(A\to B)) \to \dpr_x B)$
\item
$ \forall S\, \exists x \, \opr_\alpha(S \to \dpr_x S)$ 
\end{enumerate}

\noindent
Let ${\sf W}^+_{\alpha,\theta}$ be {\sf EA} plus the following principles.
\begin{enumerate}[A.]
\item
$ \opr_\alpha\apr A \to \opr_\alpha A$
\item
$\opr_\alpha A \to  \opr_\alpha \apr A$ 
\item  
$(\apr A \wedge \apr(A\to B)) \to \apr B$
\item
$ \forall S\,  \opr_\alpha(S \to \apr S)$ 
\end{enumerate}

\begin{theorem}\label{samenvattingsmurf}
Let $\alpha$ be a $\Delta_0({\sf exp})$-predicate that numerates the axioms of $U$ in {\sf EA}, or, equivalently, in true arithmetic.
Let $\theta$ be a $\Sigma^0_1$-predicate that satisfies ${\sf EA} \vdash (\theta(y,z) \wedge y < y') \to \theta(y',z)$.
Then,
${\sf W}_{\alpha,\theta}$ implies ${\sf W}^+_{\alpha,\theta}$.
\end{theorem}

\noindent
The logic {\sf GLT} is the bimodal propositional logic which has {\sf GL} both for $\opr$ and $\apr$,
plus the following principles.  
\begin{itemize}
\item $\vdash \apr \phi \to \opr\phi$.
\item $\vdash \opr \phi \to \apr\opr\phi$.
\item $\vdash \opr \phi \to \opr\apr \phi$.
\item $\vdash \opr\apr\phi \to \opr \phi$.
\end{itemize} 

\noindent
By Theorem~\ref{samenvattingsmurf}, we have:

\begin{theorem}\label{feestsmurf}
Let $\alpha$ be a $\Delta_0({\sf exp})$-predicate that numerates the axioms of $U$ in {\sf EA}, or, equivalently, in true arithmetic.
Let $\theta$ be a $\Sigma^0_1$-predicate that satisfies ${\sf EA} \vdash (\theta(y,z) \wedge y < y') \to \theta(y',z)$.

Suppose ${\sf W}_{\alpha,\theta}$ is a true theory and $U \vdash {\sf W}_{\alpha,\theta}$.
Then {\sf GLT} is arithmetically valid in $U$. In addition, $U$ satisfies {\sf HB} both for
$\opr_\alpha$ and for $\apr_\theta$. Finally, $\apr_\theta$ satisfies the Kreisel Condition in $U$.
\end{theorem}

\subsection{Extensions of Peano Arithmetic Revisited}\label{expa}
We show how the case of extensions of Peano Arithmetic, treated in Section~\ref{filo}, fits the framework of the present section.

Let $U$ be a consistent extension of {\sf PA} and let $\alpha$ be an elementary predicate numerating 
the axioms of $U$ in {\sf EA} with $\pi \preceq \alpha$. 
We note that, equivalently, $\alpha$  numerates the axioms of $U$ in true arithmetic. 
Let $\alpha_z(x) :\iff \alpha(x) \wedge x \leq z$.
We take 
$ \Theta_\alpha(z,x) := {\sf prov}_{\alpha_z}(x)$ in the role of $\theta$.
Thus, we have  $ \opr_{\alpha_z}A = \Theta_\alpha(z,\gnum A)$ in the role of $\dpr_z A$ and we have  
$ \widetilde\opr_\alpha A := \exists z\, (\opr_{\alpha_z}A \wedge \mathcal S(z))$  the role of $\apr A$.

We define: $\widetilde \alpha(a) :\iff \alpha (a) \wedge \mathcal S(a)$.
We have:

\begin{lemma}\label{knutselsmurf}
${\sf EA} \vdash \forall A\, (\widetilde \opr_\alpha A \iff \opr_{\widetilde\alpha} A)$. 
\end{lemma}

\begin{proof}
We reason in {\sf EA}. 

Suppose $\widetilde \opr_\alpha A$. Then, for some $z$, we have $\opr_{\alpha_z}A$ and $\mathcal S(z)$.
Suppose $p$ witnesses $\opr_{\alpha_z}A$ and $B$ is an axiom used in $p$. Then, $\alpha(B)$ and $B \leq z$.
Since $\mathcal S$ is downward persistent w.r.t. $\leq$, we find $\mathcal S(B)$, and, hence, $\widetilde\alpha(B)$.

Conversely, suppose $\opr_{\widetilde \alpha}A$. let $q$ be a witnessing proof. Let $B$ be the maximal $\alpha$-axiom used in $q$.
We find $\mathcal S(B)$. Thus, $\opr_{\alpha_B} A$ and  $\mathcal S(B)$, i.e., $\widetilde \opr_\alpha A$.
\end{proof}

\begin{lemma}
The predicate $\widetilde \alpha$ numerates the axioms of $U$ in $U$.
Hence, $\widetilde \opr_\alpha$ is Fefermanian in $U$ over $U$.
\end{lemma}

\begin{proof}
Let $X$ be the set of axioms set of axioms numerated by $\alpha$.

Suppose $n \in X$. Then $\alpha(n)$ and hence $U \vdash  \alpha(\underline n)$.
Since also, by Lemma~\ref{technischesmurf}, we have $U \vdash \mathcal S(\underline n)$,
 it follows that $U \vdash \widetilde \alpha(\underline n)$.
 
Suppose $U \vdash \widetilde\alpha (\underline n)$. Then, $U \vdash \alpha(\underline n)$ and, hence, $n \in X$. 
\end{proof}

\begin{lemma}\label{krokosmurf}
{\sf EA} verifies ${\sf W}_{\alpha,\Theta_\alpha}$.
\end{lemma}

\begin{proof}
The principle (a) follows by essential reflexivity.  The principles (b) and (c) are trivial. The principle (d) follows
since for a sufficiently large $n$ we will have, in {\sf EA}, that $\opr_{\alpha_{\underline n}} B$, where $B$ is a single axiom
for {\sf EA}.
\end{proof}

\noindent
By Theorem~\ref{feestsmurf} and Lemma~\ref{krokosmurf} we find:

\begin{theorem}
The logic {\sf GLT} is arithmetically valid for
$\opr_\alpha$ and for $\opr_{\widetilde\alpha}$ over $U$. In addition, we have {\sf HB} both for
$\opr_\alpha$ and for $\opr_{\widetilde\alpha}$ over $U$. Finally, $\opr_{\widetilde\alpha}$ satisfies the Kreisel Condition in $U$.
\end{theorem}

\noindent
In case $U$ is sound, one easily sees that the pair 
$\opr_\alpha$ and $\opr_{\widetilde \alpha}$ satisfies the conditions of Theorem~16 of \cite{henk:slow16}. It follows that
 {\sf GLT} is precisely the bi-modal propositional logic of $\opr_\alpha$ and $\opr_{\widetilde \alpha}$  in $U$, for sound $U$.

\begin{remark}\label{kooksmurf}
Let ${\sf EA}+{\sf ref}$ be {\sf EA} plus sentential reflection for predicate logic. Let $\tau$ be a standard axiomatization for
${\sf EA}+{\sf ref}$. Let $U$ be an extension of   ${\sf EA}+{\sf ref}$ and let $\alpha$ be an elementary axiomatization of $U$
such that $\tau \preceq \alpha$. With these basic ingredients we can repeat the development of the present section noting that
we are always looking at sentential reflection rather than uniform reflection.

In \cite{viss:pean14} we introduced the theory Peano Corto, which has many analogies to
${\sf EA}+{\sf ref}$. It would be interesting to see how much of our development can be repeated for the
case of Peano Corto.
\end{remark}

\section{Extensions of Elementary Arithmetic}\label{olijkesmurf}
In this section, we give a general construction of a $\theta$ with the desired properties for extensions of {\sf EA}.

 We first take a moment to see that, in order to get the desired combinations of
 properties for extensions of {\sf EA}, we indeed need to leave the realm of the Fefermanian predicates.
 
 \subsection{Two Examples}
   Our first example addresses the case that we only demand that our Fefermanian predicate is Kreiselian.
   
   \begin{example}
 Consider the theory $U := {\sf EA}+\opr_\beta\bot$. Here $\beta$ numerates a single axiom for {\sf EA}.
  Suppose there would be a Fefermanian predicate for $U$ over {\sf EA} that is Kreiselian. Say the witnessing
  predicate for the axiom set is $\alpha$.
  Let $\gamma(x) :\iff \beta(x) \vee x= \gnum{\opr_\beta\bot}$. 
  We have $\gamma \preceq \alpha$. Since, $U \vdash \opr_\beta \bot$, it follows that $U \vdash \opr_\gamma\bot$,
  and, hence, $U \vdash \opr_\alpha \bot$. So, $U \vdash \bot$. \emph{Quod non.}
  \end{example}
  
  \noindent
In the previous example, we needed an unsound theory.  In our second example, we consider the case that our example satisfies absorption. 
Here we can use a sound theory.

\begin{example}\label{gargamel}
Let $U := {\sf EA}$. Consider a Fefermanian predicate $P$ based on $\alpha$ for {\sf EA} over {\sf EA}. 
  We write $\apr$ for $P$. We note that $\beta\preceq \alpha$. Suppose we would have the absorption law for $\apr$ and $\opr_\beta$.
  Then, it would follow that:
  \begin{eqnarray*}
  {\sf EA} \vdash \opr_\beta\opr_\beta \bot & \to & \opr_\beta \apr\bot \\
  & \to & \opr_\beta \bot
  \end{eqnarray*} 
So, by L\"ob's Rule, ${\sf EA} \vdash \opr_\beta\bot$. \emph{Quod non.}
\end{example}

\begin{question}
We note that our examples are of finitely axiomatized theories. The construction of Section~\ref{expa}
gives us Fefermanian predicates for theories extending Peano Arithmetic. As pointed out in Remark~\ref{kooksmurf},
we can improve this to extensions of ${\sf EA}+{\sf ref}$. Obviously there is a big gap between our examples and 
counterexamples for the possibility to obtain a Fefermanian $\apr$. 

So, the question is whether we can find a larger class of theories for which we have a Fefermanian $\apr$ that
satisfies both the L\"ob Conditions and the Kreisel Condition and that satisfies the emission and absorption laws.  
\end{question}

\subsection{Motivating Remarks for Our Construction}
We may construct the desired predicates $\dpr_x$ in many ways. However, for didactic reasons, it good to maximize
the meaningfulness of the construction. 

As a first step, we note that we have the conditions of Theorems~\ref{grotesmurf} and \ref{stoeresmurf}, for $\opr_{\alpha,(x)}$. 
So, to obtain absorption and emission, $\dpr_x := \opr_{\alpha,(x)}$ is already sufficient. The idea of our construction is simply
to add closure under modus ponens and closure under {\sf HB} in a minimal way to $ \opr_{\alpha,(x)}$.

The minimal way to obtain the addition of modus ponens is simply to close of the $\alpha$-theorems with proofs $\leq x$
under modus ponens.  However, we can strengthen the analogy with our approach to the case of extensions of {\sf PA} by working with
a Hilbert system that only has modus ponens as a rule. Such deduction systems 
are described in \cite{quin:math96} (first edition 1940) and in \cite{fefe:arit60}.
When we have such a system we can, for the definition of $\dpr_x$, consider the theorems whose proofs contain only axioms 
\emph{whether logical or non-logical} which are $\leq x$. Thus, the main difference between our approach for the extensions of {\sf PA}
and the new one is that we stop treating logical and non-logical axioms as different.

What to do to obtain the Hilbert-Bernays condition? Simple: we add the true $\Sigma^0_1$-sentences to our original axiom set.

There is a small technical complication, due to the lack of $\Sigma^0_1$-collection, that necessitates us to stipulate a bound
on the witnesses of the truth of the $\Sigma^0_1$ sentences involved in a proof, but this complication disappears as soon as
we have  $\Sigma^0_1$-collection in the ambient theory.

\subsection{The Construction}\label{constru}
We fix a Hilbert system $\mathfrak H$ with as only rule \emph{modus ponens}. Let ${\sf logic}(x)$ be a $\Delta_0({\sf exp})$-formula that
numerates the set axioms of $\mathfrak H$ in {\sf EA}.

We assume that a $\Sigma^0_1$-sentence begins with a, possibly vacuous, existential quantifier. 

We give the basic definitions for our approach. Let a theory $U$ be given and a $\Delta_0({\sf exp})$-formula $\alpha$
that numerates the axioms of $U$ in {\sf EA} (or, equivalently, in true arithmetic).
\begin{itemize}
\item
 We define ${\sf ass}^\circ(p)$ as the set of assumptions op $p$, where now a logical axiom also counts as an assumption.
 In other words, anything not proved from previous items using modus ponens counts as an assumption.
 \item
 We write ${\sf proof}^\circ_\gamma(p,x)$ for ${\sf proof}(p,x) \wedge \forall y\in {\sf ass}^\circ(p)\, \gamma(y)$. 
 \item
 We write $\boxcircle_\gamma A$ for $\exists p\, {\sf proof}^\circ_\gamma(p,\gnum{A})$.
 \item
 $B$ is a \emph{direct $\circ$-subformula} of $A$ if $A$ is of the form $(C\to B)$ or $(B \to C)$. 
 The \emph{$\circ$-subformulas} of $A$ are the smallest set that
 contains $A$ and is closed under taking direct $\circ$-subformulas. 
 \item
 $\alpha^+(a) := \alpha(a) \vee {\sf logic}(a) \vee {\sf true}(a)$.\\ Here {\sf true} is a $\Sigma^0_1$-truth predicate. We
 we take ${\sf true}(a)$ to imply that $a$ is a $\Sigma^0_1$-sentence.
 \item
  $\widetilde\alpha^+(a) := \alpha^+(a)\wedge \mathcal S(a)$.
  \item
 $\alpha^+_{x,z}(a)  :\iff (\alpha(a) \vee {\sf logic}(a) \vee {\sf true}^z(a)) \wedge a\leq x$.
  \item
 $\alpha_x^+(a)  :\iff \alpha^+(a)  \wedge a\leq x$.
\item
$\Theta^\circ_\alpha(x,A) : \iff  \boxcircle_{\alpha^+_{x,\ast}} A : \iff \exists z\, \boxcircle_{\alpha^+_{x,z}} A $. \\
We use $\Theta^\circ_\alpha$ in the role of $\theta$. So  $\boxcircle_{\alpha^+_{x,\ast}} (\cdot)$ has the role of $\dpr_x$.
\item
We define $\widetilde \boxcircle_{\alpha^+} A :\iff \exists x\, (\boxcircle_{\alpha^+_{x,\ast}} A \wedge \mathcal S(x))$. 
So, $\widetilde \boxcircle_{\alpha^+}$ has the role of $\apr$.
\end{itemize}
 In case we have $\Sigma^0_1$-collection, the situation simplifies. We note that in the absence of $\Sigma^0_1$-collection
$ \boxcircle_{\alpha^+_{x}} A$ is not $\Sigma^0_1$ but $\Sigma^0_{1,1}$. 
See \cite{viss:pean14} or \cite{viss:orac15b} for an explanation of $\Sigma^0_{1,1}$.
 We have:
 
 \begin{lemma}
 \begin{enumerate}[a.]
 \item
 ${\sf EA} \vdash \boxcircle_{\alpha^+_{x,\ast}} A \to  \boxcircle_{\alpha^+_{x}} A$.
 \item
  ${\sf EA} + \mathrm B\Sigma_1 \vdash   \boxcircle_{\alpha^+_{x}} A \iff \boxcircle_{\alpha^+_{x,\ast}} A $.
  \item
  ${\sf EA} \vdash \widetilde \boxcircle_{\alpha^+} A \to \boxcircle_{\widetilde \alpha^+} A$.
   \item
  ${\sf EA} + \mathrm B\Sigma_1 \vdash  \boxcircle_{\widetilde \alpha^+} A \iff \widetilde \boxcircle_{\alpha^+} A $.
 \end{enumerate}
 \end{lemma}

\begin{proof}
(a) is trivial. (b) is an immediate application of collection. (c) and (d) are analogous to the proof of 
Lemma~\ref{knutselsmurf}, using respectively (a) and (b).
\end{proof}


We start with a well-known lemma.

\begin{lemma}\label{vrolijkesmurf}
 Let $\alpha$ be a $\Delta_0({\sf exp})$-predicate numerating the axiom
set of $U$ over {\sf EA}.
Then, ${\sf EA} \vdash \forall x, A \; \opr_\alpha(\opr_{\alpha,(x)} A \to A)$. 
\end{lemma}

\begin{proof}
We reason in {\sf EA}.
Suppose, for some $p\leq x$, we have ${\sf proof}_\alpha(p,A)$. It clearly follows that
$\opr_\alpha A$ and, hence, \emph{a fortiori}, $\opr_\alpha (\opr_{\alpha,(x)} A \to A)$.

Suppose, for all $q\leq x$, we have $\neg\, {\sf proof}_\alpha(q,A)$. If follows, by
$\Sigma^0_1$-completeness, that $\opr_\alpha \forall q \leq p\, \neg\, {\sf proof}_\alpha(q,A)$.
In other words, $\opr_\alpha \neg\, \opr_{\alpha,(x)}A$. If follows that $\opr_\alpha (\opr_{\alpha,(x)} A \to A)$. 
\end{proof}

\noindent
The next lemma is in the spirit of the previous one, but takes a bit more work.

\begin{lemma}\label{uitzinnigesmurf}
Let $\alpha$ be a $\Delta_0({\sf exp})$-predicate numerating the axiom
set of $U$ over {\sf EA}.
We have ${\sf EA} \vdash \forall x, A\; \opr_\alpha( \boxcircle_{\alpha^+_{x,\ast}} A \to A)$. 
\end{lemma}

\begin{proof}
 We will use a well-known fact, to wit that
\[{\sf EA} \vdash \forall x\, \opr_\alpha \forall y\, (y \leq x \iff \bigvee_{z\leq x} z=y).\]
This fact means that we do not have to worry that undesirable non-standard elements
creep in below elements that are internally standard in {\sf EA}.

We reason in {\sf EA}. Let $x$ be given.

We reason inside $\opr_\alpha$. Suppose (\$) $ \boxcircle_{\alpha^+_{x,\ast}} A$. 
Let  $z$ and $p$ witness $ \boxcircle_{\alpha^+_{x,\ast}} A$. Keeping $z$ fixed, we may, by the 
$\Delta_0({\sf exp})$-mimimum Principle, find a $p_0$ that is minimal with this property.

Suppose that $p_0$ contains a formula $B$ twice. If $B$ is the conclusion $A$ of $p_0$ we may omit the part after the first occurrence of $A$, obtaining 
a shorter proof. This contradicts the minimality of $p_0$. If $B$ is not the conclusion of  $p_0$, we may omit all occurrences of $B$ after the first one,
obtaining a shorter proof. This again contradicts the minimality of $p_0$. We may conclude that all sentences in $p_0$ occur only once in $p_0$.

We claim that every formula that is a (sub)conclusion of $p_0$ is a $\circ$-subformula of a formula in ${\sf ass}^\circ(p_0)$.
Suppose not. Let $B$ be the first such formula. Clearly, $B$ cannot be a $\circ$-assumption. So, it must be the conclusion of an application of modus
ponens and, thus, a direct $\circ$-subformula of a previous formula of the form $(C\to B)$. But this formula is by assumption a
$\circ$-subformula of ${\sf ass}^\circ(p_0)$. A contradiction.

So, all sentences occurring in $p_0$ are in $\circ$-subformulas of ${\sf ass}^\circ(p_0)$  and occur only once.
It follows that the sentences in $p_0$ are all $\leq x$ and, hence, the number of these sentences is
  also $\leq x$. So, by our assumptions on coding, we find $p_0 \approx x^x$.
  So, certainly $p_0$ will be estimated by $2^{x^2} +\underline k$, for a sufficiently large
  standard $k$. 
  
It follows that $\bigvee_{q \leq 2^{x^2} +\underline k} \, {\sf proof}^\circ_{\alpha^+}(q,A)$ and, hence,
\[(\dag)\;\; \bigvee_{q \leq 2^{x^2} +\underline k} \, {\sf proof}_{\alpha \cup {\sf true}}(q,A).\] (Here the $q$ are standard
on the $\opr_\alpha$-external {\sf EA}-level.) 

Now, suppose (\ddag) ${\sf proof}_{\alpha \cup {\sf true}}(q,A)$, where $q\leq 2^{x^2} +\underline k$. We transform $q$ as follows. Let $\mathscr S$ be the
set of the $\Sigma^0_1$-sentences in ${\sf ass}(q)$ that are not in $\alpha$. It follows that all $S\in \mathscr S$ are {\sf true}.
We transform $q$ in two steps. First we form a proof $q'$ from the assumptions $({\sf ass}(q) \setminus \mathscr S ) \cup \bigwedge \mathscr S$ 
with conclusion $A$. Then, we transform $q'$ to $q''$ with assumptions ${\sf ass}(q) \setminus \mathscr S$ to $\bigwedge \mathscr S \to A$.
We note that the big conjunction is bounded by $q$ and, thus, exists at the $\opr_\alpha$-external level. 

We easily see that $|q'|$ can be bounded by a linear term in $|q|$. The transformation $q' \mapsto q''$ uses the deduction theorem.
Inspection of the proof shows that here also $|q''|$ is linear in $|q'|$. Thus, $q''$ is bounded by $2^{\underline m x^2} +\underline n$, for appropriate
standard $m$ and $n$. We conclude that $q''$ is also $\opr_\alpha$-external. We have found that $\opr_{\alpha,(q'')} (\bigwedge \mathscr S \to A)$,
where $q''$ is $\opr_\alpha$-external.

We apply Lemma~\ref{vrolijkesmurf} to obtain $\bigwedge \mathscr S \to A$. We also have, since all elements of $\mathcal S$ are
{\sf true} and $\mathcal S$ is $\opr_\alpha$-external, that $\bigwedge\verz{{\sf true}(S) \mid S \in \mathscr S}$.
From this it follows that $\bigwedge \mathcal S$.
Combining $\bigwedge \mathscr S \to A$ and $\bigwedge \mathcal S$, we find $A$.
  
  By (\dag) we find $A$ without assumption (\ddag). We now cancel (\$) to obtain the sentence: $ \boxcircle_{\alpha^+_{x,\ast}} A \to A$.

  We leave the $\opr_\alpha$-environment. We have shown $\opr_\alpha ( \boxcircle_{\alpha^+_{x,\ast}} A \to A)$, as desired.
\end{proof}

\noindent
We insert a quick corollary of Lemma~\ref{uitzinnigesmurf}.

\begin{corollary}
Let $\alpha$ be a $\Delta_0({\sf exp})$-predicate numerating the axiom
set of $U$ over {\sf EA}. Then,
${\sf EA} \vdash \forall A\; (\exists x\, \boxcircle_{\alpha^+_{x,\ast}} A \iff \opr_\alpha A)$.
\end{corollary}

\begin{proof}
We reason in {\sf EA}. The left-to-right direction works as follows. We  use Lemma~\ref{uitzinnigesmurf}.
\begin{eqnarray*}
\exists x\, \boxcircle_{\alpha^+_{x,\ast}} A & \to & \exists x\, \opr_\alpha  \boxcircle_{\alpha^+_{x,\ast}} A \\
& \to & \opr_\alpha A
\end{eqnarray*} 

\noindent
The right-to-left direction is immediate since 
$\opr_\alpha A$ implies $\opr_{\alpha, (x)} A$, for some $x$, and 
$\opr_{\alpha, (x)} A$ implies  $\boxcircle_{\alpha^+_{x,\ast}} A$.
\end{proof}

\begin{lemma}\label{alligatorsmurf}
Let $\alpha$ be a $\Delta_0({\sf exp})$-predicate numerating the axiom
set of $U$ over {\sf EA}. Then,
{\sf EA} verifies ${\sf W}_{\alpha,\Theta^\circ_\alpha}$.
\end{lemma}

\begin{proof}
The principle (a) follows by Lemma~\ref{uitzinnigesmurf}.  
The principle (b) follows by: 
\begin{eqnarray*}
{\sf EA} \vdash \opr_\alpha A & \to & \exists x\, \opr_\alpha (\opr_{\alpha,(x)} A \wedge \mathcal S(x)) \\
& \to &  \exists x \, \opr_\alpha (\boxcircle_{\alpha^+_{x,\ast}} A \wedge \mathcal S(x)) 
\end{eqnarray*}

\noindent
The principles (c) and (d) are immediate by the construction of $\boxcircle_{\alpha^+_{x,\ast}}$.
\end{proof}

\noindent
By Theorem~\ref{feestsmurf} and Lemma~\ref{alligatorsmurf} we find:

\begin{theorem}\label{dreumessmurf}
Let $\alpha$ be a $\Delta_0({\sf exp})$-predicate numerating the axiom
set of $U$ over {\sf EA}. Then,
the logic {\sf GLT} is arithmetically valid in $U$ for  $\widetilde\boxcircle_{\alpha^+}$ and $\opr_\alpha$. 
In addition, we have {\sf HB} over $U$ both for
$\opr_\alpha$ and for $\opr_{\widetilde\alpha}$. Finally, $\widetilde\boxcircle_{\alpha^+}$ satisfies the Kreisel Condition in $U$.
\end{theorem}

\begin{question}
The predicate $\boxcircle_{\alpha^+_{x,\ast}}$ allows us to give an Orey H\'ajek Characterization for extensions of {\sf EA}. 
Suppose $\gamma$ is an elementary predicate that numerates the axioms of $V$ over {\sf EA}. Then,
$U$ is $\Pi^0_1$-conservative over $V$ iff, for all $n$, we have $U\vdash \odotpossible_{\gamma^+_{\underline n,\ast}}\top$. 

It seems to me that, using this characterization, many results in the work of Per Lindstr\"om and Christian Bennet should be transferable
from the case of extensions of {\sf PA} to the case of extensions of {\sf EA}. It would be interesting to explore this.
\end{question}

\begin{question}
In our definition of $\boxcircle$, we counted instances of all  proposition-logical and  predicate-logical 
schemes as axioms. I think we only need to count the instances of a few specific schemes that are
essentially predicate logical. It would be interesting to explore this.
\end{question}

\section{An Application}
We first prove a very general result.

\begin{theorem}\label{conservatievesmurf}
Let $U$ be any theory and suppose that $\apr$ satisfies  the L\"ob Conditions 
in $U$ and {\sf HB}, to wit, $U \vdash S \to \apr S$, for $S\in \Sigma^0_1$.\footnote{We note that in the present general
 context
the L\"ob Conditions and {\sf HB} are mutually independent.} We have:
\begin{enumerate}[i.]
\item
$U$ is $\Pi_1^0$-conservative over $U+\apr \bot$.
\item
If $\apr$ satisfies the Kreisel condition for $U$, then
$U$ is $\Sigma_1^0$-conservative over $U+\neg \, \apr \bot$.
\end{enumerate}
\end{theorem}

\begin{proof}
Suppose $\apr$ satisfies the L\"ob conditions and the {\sf HB} for $U$.

We prove (i). 
Let $P$ be a $\Pi^0_1$-sentence. Suppose $U + \apr\bot \vdash P$. Then,
(a) $U + \neg\, P \vdash  \neg \, \apr\bot$. Hence, $U \vdash \apr \neg\, P \to \apr \neg\, \apr \bot$.
It follows by {\sf HB} and by the formalized Second Incompleteness Theorem for $\apr$,
which follows by the L\"ob Conditions,  that
(b) $U + \neg\, P \vdash  \apr\bot$. Combining (a) and (b), we find $U \vdash P$.

We prove (ii). Suppose that $\apr$ satisfies the Kreisel Condition for $U$.
Let $S$ be a $\Sigma^0_1$-sentence. 
Suppose $U + \neg\, \apr \bot \vdash S$. It follows that
$U \vdash \apr\bot \vee S$, and, hence, by  the L\"ob Conditions and {\sf HB},
$U \vdash \apr S$. By the Kreisel Condition, we find $U \vdash S$.
\end{proof}

\noindent
The proof of (i) is ascribed by Per Lindstr\"om, in \cite[p94]{lind:aspe03}, to Georg Kreisel in \cite{kreis:weak62}. 

Consider any recursively enumerable theory $U$ and let $\alpha$ be a $\Delta_0({\sf exp})$-formula  
that numerates a set of axioms for $U$ in {\sf EA}.
We note that, by Craig's trick, we can always find  such a $\Delta_0({\sf exp})$-formula.
 We take $\apr_\alpha :=\widetilde\boxcircle_{\alpha^+}$.
We note that $\apr$ fulfills the conditions of Theorem~\ref{conservatievesmurf} for $U$.
It follows that $\apr_\alpha \bot$ is a $\Sigma^0_1$-sentence such that $U$ is $\Pi^0_1$-conservative
over $U+\apr_\alpha \bot$ and  $\neg\,\apr_\alpha \bot$ is a $\Pi^0_1$-sentence such that $U$ is $\Sigma^0_1$-conservative
over $U+\apr_\alpha \bot$. 

For extensions $U$ of Peano Arithmetic, the existence of a $\Sigma^0_1$-sentence $S$, such that $U$ is $\Pi^0_1$-conservative over
$U+S$ and $U$ is $\Sigma^0_1$-conservative over
$U+\neg \,S$ is a special case of a result due to Robert Solovay. See
\cite{guas:part79}. See also \cite[Chapter 5]{lind:aspe03}.\footnote{Our method can be extended to prove the full Solovay result.}

\begin{remark}
We note that $\apr_\alpha\bot$ is \emph{a fortiori} a Rosser sentence for $U$.
The resulting proof of Rosser's Theorem is like the proof of the Second Incompleteness Theorem
in the sense that the sentence under consideration is self-reference-free, but in the proof
of the desired property we do use self-reference. 
Another example of a self-reference-free $\Sigma^0_1$ Rosser sentence (for extensions of {\sf PA}) is due to Fedor Pakhomov.
See \cite{pakh:sol217}. We note that Pakhomov's construction is, in a sense, orthogonal to ours. An essential feature of Pakhomov's construction is
that, like the ordinary Rosser sentence and its opposite,  it produces $\Sigma^0_1$-sentences $S_0$ and $S_1$, each with the Rosser property over $U$, such that
we have $U \vdash \neg\, (S_0\wedge S_1)$ and $U \vdash \opr_\alpha \bot \iff (S_0\vee S_1)$. It follows that e.g. $U \vdash S_0 \to \neg\, S_1$, but $U\nvdash \neg\, S_1$.
So, $S_0$ is not $\Pi^0_1$-conservative. The non-$\Pi^0_1$-conservativity of Pakhomov's sentences is an important feature since it allows him to use
them for his alternative proof of Solovay's arithmetical completeneness theorem for L\"ob's Logic.
\end{remark}

\noindent
We formulate a consequence of Theorem~\ref{sluwesmurf}.

\begin{theorem}\label{nozelsmurf}
Suppose $U$ is a recursively enumerable extension of {\sf PA}.
Then there is a $\Sigma^0_1$-predicate
$\sigma^\ast$ such that $\tupel{U,U,\sigma^\ast}$ is Fefermanian and $\opr_{\sigma^\ast}$ satisfies, in $U$, the L\"ob-conditions and the
Kreisel condition.
\end{theorem} 

\begin{proof}
Consider a recursively enumerable extension $U$ of {\sf PA}. We can easily construct a $\Sigma^0_1$-formula $\sigma$ such that 
$\tupel{{\sf PA},U,\sigma}$ is Fefermanian and $\pi\preceq \sigma$. We now apply Theorem~\ref{sluwesmurf} to obtain the desired $\sigma^\ast$.
\end{proof}

\noindent It follows from Theorems~\ref{conservatievesmurf}(ii) and \ref{nozelsmurf} that:

\begin{theorem}\label{smurferella}
Suppose $U$ is a recursively enumerable extension of {\sf PA}. 
Then, there is a $\Sigma^0_1$-predicate
$\sigma^\ast$ such that $\tupel{U,U,\sigma^\ast}$ is Fefermanian and $\oco_{\sigma^\ast}\top$ is $\Sigma^0_1$-conservative over $U$.
\end{theorem}

\noindent
Theorem~\ref{smurferella} stands in interesting contrast to a result due to Craig Smory\'nski that
is reported in Example~1.6 of \cite{guas:part79}. 

\begin{theorem}[Smory\'nski]\label{smoorf}
Suppose $U$ is an extension of {\sf PA} such  $\tupel{{\sf PA},U,\sigma}$ is Fefermanian and
$\alpha$ is $\Delta_1^0({\sf PA})$. Then, $\oco_\alpha\top$ is $\Sigma^0_1$-conservative over
$U$ iff $U$ is $\Sigma^0_1$-sound. 
\end{theorem}

\noindent
Inspection of the proof shows that  no special properties of {\sf PA} are  used (except its $\Sigma^0_1$-soundness
when it occurs in the role of base theory), so, in fact, we have a far more
general result. However, we will not pursue that line here.
We just present a variation of Smory\'nski's result for extensions of {\sf PA}.

\begin{theorem}\label{smoorsmurf}
Suppose $\tupel{{\sf PA},U,\sigma}$ is Fefermanian, where $\sigma$ is $\Sigma^0_1$. Then, 
the following are equivalent.
\begin{enumerate}[i.]
\item
$U$ is $\Sigma^0_1$-sound.
\item
$\opr_\sigma$ satisfies the Kreisel condition for $U$.
\item
$\oco_{\sigma}\top$ is $\Sigma^0_1$-conservative over $U$.
\end{enumerate}
\end{theorem}

\begin{proof}
Suppose $\tupel{{\sf PA},U,\sigma}$ is Fefermanian, where $\sigma$ is $\Sigma^0_1$.

(i) $\To$ (ii). Suppose $U$ is $\Sigma^0_1$-sound.
We note that the fact that $\tupel{{\sf PA},U,\sigma}$ is Fefermanian, implies
that $\sigma$ truly enumerates the G\"odel codes of an axiom set of $U$.
Suppose $U \vdash \opr_\sigma A$. Then, $\opr_\sigma A$ is true, and, hence $U \vdash A$.

(ii) $\To$ (iii). This is immediate by Theorem~\ref{conservatievesmurf}.

(iii) $\To$ (i). Suppose  $U+\oco_{\sigma}\top$ is $\Sigma^0_1$-conservative over
$U$. By Theorem~\ref{labbekaksmurf}, we have $U\vdash \forall A\,  (\opr_{\hat\sigma}A \iff \opr_\sigma A)$.
So, $U+\oco_{\hat\sigma}\top$ is $\Sigma^0_1$-conservative over
$U$.  By Theorem~\ref{kleutersmurf}, the triple $\tupel{{\sf PA},U, \hat\sigma}$ is Fefermanian.
Moreover, $\hat\sigma$ is elementary and, hence, \emph{a fortiori}, $\Delta^0_1({\sf PA})$.
So, by Theorem~\ref{smoorf}, we find that $U$ is $\Sigma^0_1$-sound.
\end{proof}

The contrast between Theorem~\ref{smurferella} and Theorem~\ref{smoorsmurf},
 illustrates well how sensitive the choice of the base theory may be already in the 
 case where we consider $\Sigma^0_1$-axiomatizations. Moreover, this contrast
 may be a warning against uncritical use of Craig's trick of the form \emph{we may always replace
 a $\Sigma^0_1$-axiomatization by a $\Delta_0({\sf exp})$-axiomatization}.
 This is only unproblematic when we have a $\Sigma^0_1$-sound base theory and
 consider extensions of ${\sf EA}+\mathrm B\Sigma^0_1$. In all other cases, some care
 is needed.
   

\appendix

\section{Examples}\label{examples}
For convenience, we repeat our overview of the examples.

\[
\begin{tabular}{|l||c|c|c||c|c|c|} \hline
& base & lead & $P$ & L\"ob & Kreisel & Feferman \\ \hline\hline
Example~\ref{ppp} & {\sf EA} & {\sf EA} & $\Sigma^0_1$ & $+$ & $+$ & $+$ \\ \hline
Example~\ref{ppm} &{\sf EA} & {\sf EA} & $\Sigma^0_1$ & $+$ & $+$ & $-$ \\ \hline
Example~\ref{pmp} & {\sf EA} & ${\sf EA}+\opr_\beta\bot$ & $\Sigma^0_1$ & $+$ & $-$ & $+$ \\ 
  & {\sf EA} & {\sf EA} & $\Sigma^0_2$ &  &  &  \\ \hline
Example~\ref{pmm} &{\sf EA} &{\sf EA} & $\Sigma^0_1$ & $+$ & $-$ & $-$ \\ \hline
Example~\ref{mpp} &{\sf PA} & {\sf PA} & $\Sigma^0_2$ & $-$ & $+$ & $+$ \\ 
& {\sf EA} & {\sf EA} & $\Sigma^0_{1,1}$ &&& \\ \hline
Example~\ref{mpm} &{\sf EA} & {\sf EA} & $\Sigma^0_1$ & $-$ & $+$ & $-$ \\ \hline
Example~\ref{mmp} & {\sf PA} & {\sf PA} & $\Sigma^0_2$ & $-$ & $-$ & $+$ \\ 
 & {\sf EA} & ${\sf EA}+\opr_\beta\opr_\pi\bot$ & $\Sigma^0_{1,1}$ &  &  &  \\ \hline
Example~\ref{mmm} & {\sf EA} & {\sf EA} & $\Sigma^0_1$ & $-$ & $-$ & $-$ \\ \hline
\end{tabular}
\]

\begin{example}\label{ppp} ${+}{+}{+}$:
We take $U_0 := U := {\sf EA}$ and $P := {\sf prov}_\beta$. 
Clearly, this $P$ satisfies all three conditions for {\sf EA}.
We note that our example satisfies {\sf HB} too.
\end{example}

\begin{example}\label{ppm}${+}{+}{-}$:
By Theorem~\ref{dreumessmurf}, there is a $\Sigma^0_1$-predicate $P$ for $U := {\sf EA}$, that satisfies L\"ob Conditions and the absorption principle
in combination with $\opr_\beta$. By Example~\ref{gargamel}, the predicate $P$ cannot be Fefermanian.

Here is a second example. We consider 
${\sf cfprov}_\beta(x)$, which stands for cut-free provability in {\sf EA}. 
Let's write $\opr^{\sf cf}_\beta A$ for ${\sf cfprov}(\gnum{A})$.
We have L\"ob's Logic for $\opr^{\sf cf}_\beta$. See \cite{viss:inte90} and \cite{kals:towa91}. 
Also we easily see that $\opr^{\sf cf}_\beta$ satisfies the Kreisel condition.
However, $\opr^{\sf cf}_\beta$ cannot be Fefermanian for {\sf EA} over {\sf EA}. 
If it were Fefermanian, we would have ${\sf EA} \vdash \opr_\beta \bot \to \opr_\beta^{\sf cf}\bot$.
To prove that this is impossible is outside the scope of the present article. We just give the outline 
of the proof, so that the reader can see the basic idea.

Suppose ${\sf EA} \vdash \opr_\beta \bot \to \opr_\beta^{\sf cf}\bot$. It follows that ${\sf EA} \vdash \oco^{\sf cf}_\beta \top \to \oco_\beta\top$.
Then, by a meta-theorem from \cite{wilk:sche87}, it follows that (a) ${\sf S}^1_2 + \oco^{\sf cf}_\beta\top \vdash \oco_\beta^J\top$, for a definable cut $J$.
We also have that (b)  {\sf EA} interprets ${\sf S}^1_2 + \oco^{\sf cf}_\beta\top$. Combining (a) and (b), we find that {\sf EA} interprets
${\sf S}^1_2+\oco_\beta\top$. But this contradicts the Second Incompleteness Theorem.

We note that our examples also satisfy {\sf HB}. A disadvantage is that they do not work for the global version
of the L\"ob Conditions, where the quantifiers over sentences for {\sf L}2 and {\sf L}3 are inside the theory. It would be interesting to have an example
for this case.  
\end{example} 

\begin{example}\label{pmp}${+}{-}{+}$:
 Here is an example of a Fefermanian predicate that does satisfy the L\"ob Conditions and does not satisfy the Kreisel
 Condition. Let $U_0:= {\sf EA}$,  $U:={\sf EA}+\opr_\beta \bot$. Let $P := {\sf prov}_{\gamma}$, where $\gamma(x) :\iff \beta(x) \vee x = \gnum{\opr_\beta\bot}$. 

We have
$U\vdash \opr_\gamma \bot$, but  $U\nvdash \bot$, so  the Kreisel condition fails for $P$ and $U$.

We note that our example also works for $U_0 := U:={\sf EA}+\opr_\beta\bot$.

We provide a second example, where the base and the lead theories are sound. By Theorem~\ref{tuinsmurf},
the predicate that represents the axioms cannot be $\Sigma^0_1$.

We take $U_0 := U := {\sf EA}$.
We define:
$ \delta(x) :\iff \beta(x) \vee (\oco_\beta \top \wedge x = \gnum{\bot})$.
We note that $\delta$ numerates $\verz{B}$ in {\sf EA}, where $B$ is the single axiom for {\sf EA}. 
We find:
\begin{eqnarray*}
{\sf EA} \vdash \opr_\delta\bot & \iff & (\opr_\beta \bot \wedge \opr_\beta \bot) \vee (\oco_\beta\top \wedge \opr_\beta (\bot \to \bot)) \\
& \iff & \top 
\end{eqnarray*}
So ${\sf EA} \vdash \opr_\delta \bot$. It follows that ${\sf prov}_\delta$ is not Kreiselian and 
satisfies the Feferman Conditions and the L\"ob Conditions.  
We note that ${\sf prov}_\delta$ is $\Sigma^0_2$.
\end{example} 

\begin{example}\label{pmm}${+}{-}{-}$:
Let $U_0 := U := {\sf EA}$ and let $P$ be $x=x$. Clearly, $P$ satisfies the L\"ob conditions in {\sf EA}, but $P$ is
not Kreiselian. Since {\sf EA} is sound and $P$ is $\Sigma^0_1$, a Fefermanian $P$ must be Kreiselian. So, $P$ is also not Fefermanian.  
\end{example}

\begin{example}\label{mpp}${-}{+}{+}$:
The case of Fefermanian predicates that do not satisfy the L\"ob Conditions is among the most interesting of our cases. The study of the possibilities for such
predicates for the case of extensions of Peano Arithmetic has been taken up by Taishi Kurahashi in great depth. See \cite{kura:arit18a} and \cite{kura:arit18b}.

A classical example of such a predicate is \emph{Feferman Provability}. 
 We define $\pi^\star(y) :\iff  \exists x \,(\pi_x(y) \wedge \oco_{\pi_x}\top)$.  Let $P := {\sf prov}_{\pi^\star}$. 
 This predicate was introduced by Solomon Feferman in his classical paper \cite{fefe:arit60}.
By the essential reflexivity of {\sf PA}, one finds that $\opr_{\pi^\star}$ is Fefermanian for {\sf PA} over {\sf PA}. 
For closely related reasons $\opr^\star$ is Kreiselian. However, $\opr^\star$ does not satisfy the 
L\"ob Conditions. The bimodal provability logic of $\opr_\pi$ and $\opr_{\pi^\star}$
 has been characterized by Volodya Shavrukov in \cite{shav:smar94}. For some earlier work, see \cite{mont:alge78} and \cite{viss:pean89}.
We note that $\opr_{\pi^\star}$ is $\Sigma^0_2$.

An example of quite different flavor uses the fact that {\sf EA} does not verify $\Sigma^0_1$-collection. We refer the reader to \cite[Subsection 6.2]{viss:look19}.
This example provides a $\Sigma^0_1$-axiomatization $\sigma$. As a consequence ${\sf prov}_\sigma$ is $\Sigma^0_{1,1}$. We refer the reader to e.g.
\cite{viss:pean14} for a further explanation of the relevant formula hierarchy.
\end{example}

\begin{example}\label{mpm}${-}{+}{-}$:
Here is an example of a $P$ that satisfies the Kreisel Condition but not the L\"ob Conditions and the Feferman Condition.
Let $U_0 := U:= {\sf EA}$. Let
$P(x) := ({\sf prov}_\pi(x) \wedge x \neq\, \gnum{\bot})$. 

We note that {\sf L}2 fails for $P$ over {\sf EA}. This shows that $P$ does not satisfy the L\"ob Conditions and, hence, cannot be Fefermanian.
\end{example}

\begin{example}\label{mmp}${-}{-}{+}$: 
The examples are adaptations of the predicates and theories in Example~\ref{mpp}. We just add something to make 
the examples non-Kreiselian. We use the notations of Example~\ref{mpp}. 

We give our first example of a non-Kreiselian Fefermanian predicate that does not satisfy the L\"ob Conditions.
Let $U_0:= U := {\sf PA}$.
We take: \[\pi^\circ(x) :\iff \pi^\star(x) \vee (\oco_\pi\top \wedge x = \gnum{\opr_\pi\bot}).\] 
Let $P := {\sf prov}_{\pi^\circ}$.

It is easily seen that $\pi^\circ$ numerates the axioms of {\sf PA} in {\sf PA}. 
We have, using the fact that we have {\sf HB} for $\opr_{\pi^\star}$:
\begin{eqnarray*}
{\sf PA} \vdash \opr_{\pi^\circ} \opr_\pi\bot & \iff & (\opr_\pi\bot \wedge \opr_{\pi^\star}\opr_\pi\bot) \vee (\oco_\pi \top \wedge \opr_{\pi^\star}(\opr_\pi\bot\to\opr_\pi\bot)) \\
& \iff & \opr_\pi\bot \vee \oco_\pi\top \\
& \iff & \top
\end{eqnarray*}
It follows that ${\sf PA} \vdash \opr_{\pi^\circ}\opr_\pi \bot$. However, ${\sf PA} \nvdash \opr_\pi\bot$, so ${\sf prov}_{\pi^\circ}$ is not Kreiselian.
We have:
\begin{eqnarray*}
{\sf PA} \vdash \opr_{\pi^\circ} \bot & \iff & (\opr_\pi\bot \wedge \opr_{\pi^\star}\bot) \vee (\oco_\pi \top \wedge \opr_{\pi^\star}\neg\, \opr_\pi\bot) \\
& \iff & (\opr_\pi\bot \wedge \bot) \vee (\oco_\pi \top \wedge \opr_{\pi}\neg\, \opr_\pi\bot)  \\
& \iff & \oco_\pi \top \wedge \opr_{\pi}\bot \\
& \iff &\bot
\end{eqnarray*}
So ${\sf PA} \vdash \neg\, \opr_{\pi^\circ} \bot$. Thus, $\opr_{\pi^\circ}$ cannot satisfy the L\"ob Conditions.
Finally, $\opr^\circ$ is clearly $\Sigma^0_2$.

\medent
Here is our second example. The presentation of our example presupposes that the reader has \cite[Subsection 6.2]{viss:look19} at hand.
The predicate $\sigma$ is imported here from that paper.
We take $U_0:={\sf EA}$, $U := {\sf EA}+\opr_\beta\opr_\pi\bot$. We define
$\sigma^\circ(x) :\iff \sigma(x) \vee x= \gnum{\opr_\beta\opr_\pi\bot}$. Clearly, $\sigma^\circ$ numerates the axioms of $U$ in {\sf EA}. 
We take $P(x) := {\sf prov}_{\sigma^\circ}(x)$. Evidently, $P$ is Fefermanian for ${\sf EA}+\opr_\beta\opr_\pi\bot$ over {\sf EA}.

 Since, we have ${\sf EA} \vdash \opr_\beta C \to \opr_\sigma C$ and   ${\sf EA} \vdash \opr_\sigma C \to \opr_{\sigma^\circ} C$.
 We find $U \vdash \opr_{\sigma^\circ} \opr_\pi\bot$. Suppose we would have $U \vdash \opr_\pi\bot$. In would follow that
 ${\sf EA} + \opr_\beta \opr_\pi\bot \vdash \opr_\pi\bot$, and, hence, ${\sf EA} \vdash \opr_\pi\bot$. \emph{Quod non}.
 Thus $U \nvdash \opr_\pi\bot$. So, $P$ is not Kreiselian.
 
 We note that over $U$ we have, by $\Sigma^0_1$-completeness, that $\opr_\sigma$ and $\opr_{\sigma^\circ}$ coincide also
 in iterated $\opr_\sigma$-contexts.
 Suppose $\opr_{\sigma^\circ}$ satisfies the L\"ob Conditions over $U$. It follows that $\opr_{\sigma}$ also satisfies
 the L\"ob Conditions over $U$. So, \emph{a fortiori}, we find $U \vdash \opr_\sigma \oco_\sigma \top \to \opr_\sigma \bot$.
 By Lemma~6.11 and Lemma~6.12  of \cite[Subsection 6.2]{viss:look19}, we find:
 \[ {\sf EA} + \opr_\beta\opr_\pi\bot \vdash ({\sf S}^\star \vee \opr_\beta\bot) \to (({\sf S}^\star \wedge \opr_\beta\opr_\beta \bot) \vee \opr_\beta\bot).\]
 It follows that
 ${\sf EA} +\opr_\pi\bot + {\sf S}^\star \vdash  \opr_\beta\opr_\beta \bot$. However, we can construct a model of 
 ${\sf EA}+\opr_\pi\bot + {\sf S}^\star+ \neg\opr_\beta\opr_\beta\bot$ using
 the construction described in \cite[Subsection 6.2]{viss:look19}.
\end{example}

\begin{example}\label{mmm}${-}{-}{-}$:
We take $U_0 := U := {\sf EA}$ and $P(x) := \bot$. 
It is clear that $P$ does not satisfy the L\"ob Conditions. Nor is it Kreiselian or Fefermanian.
\end{example}

\end{document}